\documentclass[unknownkeysallowed, 10pt, a4]{amsart}
\usepackage[utf8]{inputenc,hyperref}
\usepackage{amsaddr}
\usepackage{amsmath, amsthm}
\usepackage[maxnames=4,maxalphanames=4,style=alphabetic-verb]{biblatex}
\usepackage{mathtools}
\usepackage{amssymb}
\bibliography{main.bib}
    \usepackage{enumitem}
    \usepackage{comment}
    \usepackage{bm}
\usepackage[margin=1in]{geometry}
\usepackage{quiver}
\usepackage{graphicx} 
\usepackage{tikz-cd}
\usepackage{bbm}
\usepackage{enumitem}
\usepackage{xcolor}
\usepackage{tikz}
\usepackage[linguistics]{forest}

\hypersetup{
colorlinks,
citecolor=blue,
filecolor=green,
linkcolor=magenta,
urlcolor=blue}

\makeatletter
\newcommand{\otherlabel}[2]{\protected@edef\@currentlabel{#2}\label{#1}}
\makeatother
\usepackage{xparse}
\usepackage{epigraph}
\usepackage{float}
\usepackage{etoolbox}
\makeatletter
\patchcmd{\epigraph}{\@epitext{#1}}{\itshape\@epitext{#1}}{}{}
\makeatother
\usetikzlibrary{automata, positioning, arrows}
\tikzset{->, >=latex,  node distance=2cm,  every state/.style={thick, fill=gray!10},  initial text=~~$ $}
\usepackage{subfigure}
\usepackage{color}
\usepackage{caption}
\usepackage{mathdots}
\title[Automating stable rank computation]{{Automating the stable rank computation for special biserial algebras}\vspace{-6mm}}
\author[Srivastava and Kuber]{Suyash Srivastava and Amit Kuber}
\address{Department of Mathematics and Statistics\\Indian Institute of Technology, Kanpur\\ Uttar Pradesh, India}
\email{suyashsriv20@gmail.com, askuber@iitk.ac.in}
\keywords{special biserial algebra, radical of the module category, automaton, hammock, word problem}
\subjclass[2020]{16S90, 68Q45, 06A05, 16G20}

\DeclareLabeldate[misc]{
  \field{date}
  \field{year}
  \field{eventdate}
  \field{origdate}
  \field{urldate}
}

\begin{document}
 \def\fileversion{0.32}
 \def\filedate{2012/12/11}
\message{`dirtree' v\fileversion, \filedate\space (jcc)}
\edef\DTAtCode{\the\catcode`\@}
\catcode`\@=11
\long\def\DT@loop#1\DT@repeat{%
  \def\DT@iterate{#1\relax\expandafter\DT@iterate\fi}%
  \DT@iterate
  \let\DT@iterate\relax
}
\let\DT@repeat=\fi
\expandafter\ifx\csname DT@fromsty\endcsname\relax
  \def\@namedef#1{\expandafter\def\csname #1\endcsname}
  \def\@nameuse#1{\csname #1\endcsname}
  \long\def\@gobble#1{}
\fi
\def\@nameedef#1{\expandafter\edef\csname #1\endcsname}
\newdimen\DT@offset \DT@offset=0.2em
\newdimen\DT@width \DT@width=1em
\newdimen\DT@sep \DT@sep=0.2em
\newdimen\DT@all
\DT@all=\DT@offset
\advance\DT@all \DT@width
\advance\DT@all \DT@sep
\newdimen\DT@rulewidth \DT@rulewidth=0.4pt
\newdimen\DT@dotwidth \DT@dotwidth=1.6pt
\newdimen\DTbaselineskip \DTbaselineskip=\baselineskip
\newcount\DT@counti
\newcount\DT@countii
\newcount\DT@countiii
\newcount\DT@countiv
\def\DTsetlength#1#2#3#4#5{%
  \DT@offset=#1\relax
  \DT@width=#2\relax
  \DT@sep=#3\relax
  \DT@all=\DT@offset
  \advance\DT@all by\DT@width
  \advance\DT@all by\DT@sep
  \DT@rulewidth=#4\relax
  \DT@dotwidth=#5\relax
}
\expandafter\ifx\csname DT@fromsty\endcsname\relax
  \def\DTstyle{\tt}
  \def\DTstylecomment{\rm}
\else
  \def\DTstyle{\ttfamily}
  \def\DTstylecomment{\rmfamily}
\fi
\def\DTcomment#1{%
  \kern\parindent\dotfill
  {\DTstylecomment{#1}}%
}
\newdimen\DT@indent
\newdimen\DT@parskip
\newdimen\DT@baselineskip
\def\dirtree#1{%
  \DT@indent=\parindent
  \parindent=\z@
  \DT@parskip=\parskip
  \parskip=\z@
  \DT@baselineskip=\baselineskip
  \baselineskip=\DTbaselineskip
  \let\DT@strut=\strut
  \def\strut{\vrule width\z@ height0.7\baselineskip depth0.3\baselineskip}%
  \DT@counti=\z@
  \let\next\DT@readarg
  \next#1\@nil
  \dimen\z@=\hsize
  \advance\dimen\z@ -\DT@offset
  \advance\dimen\z@ -\DT@width
  \setbox\z@=\hbox to\dimen\z@{%
    \hsize=\dimen\z@
    \vbox{\@nameuse{DT@body@1}}%
  }%
  \dimen\z@=\ht\z@
  \advance\dimen0 by\dp\z@
  \advance\dimen0 by-0.7\baselineskip
  \ht\z@=0.7\baselineskip
  \dp\z@=\dimen\z@
  \par\leavevmode
  \kern\DT@offset
  \kern\DT@width
  \box\z@
  \endgraf
  \DT@countii=\@ne
  \DT@countiii=\z@
  \dimen3=\dimen\z@
  \@namedef{DT@lastlevel@1}{-0.7\baselineskip}%
  \loop
  \ifnum\DT@countii<\DT@counti
    \advance\DT@countii \@ne
    \advance\DT@countiii \@ne
    \dimen\z@=\@nameuse{DT@level@\the\DT@countii}\DT@all
    \advance\dimen\z@ by\DT@offset
    \advance\dimen\z@ by-\DT@all
    \leavevmode
    \kern\dimen\z@
    \DT@countiv=\DT@countii
    \count@=\z@
    \DT@loop
      \advance\DT@countiv \m@ne
      \ifnum\@nameuse{DT@level@\the\DT@countiv} >
        \@nameuse{DT@level@\the\DT@countii}\relax
      \else
        \count@=\@ne
      \fi
    \ifnum\count@=\z@
    \DT@repeat
    \edef\DT@hsize{\the\hsize}%
    \count@=\@nameuse{DT@level@\the\DT@countii}\relax
    \dimen\z@=\count@\DT@all
    \advance\hsize by-\dimen\z@
    \setbox\z@=\vbox{\@nameuse{DT@body@\the\DT@countii}}%
    \hsize=\DT@hsize
    \dimen\z@=\ht\z@
    \advance\dimen\z@ by\dp\z@
    \advance\dimen\z@ by-0.7\baselineskip
    \ht\z@=0.7\baselineskip
    \dp\z@=\dimen\z@
    \@nameedef{DT@lastlevel@\the\DT@countii}{\the\dimen3}%
    \advance\dimen3 by\dimen\z@
    \advance\dimen3 by0.7\baselineskip
    \dimen\z@=\@nameuse{DT@lastlevel@\the\DT@countii}\relax
    \advance\dimen\z@ by-\@nameuse{DT@lastlevel@\the\DT@countiv}\relax
    \advance\dimen\z@ by0.3\baselineskip
    \ifnum\@nameuse{DT@level@\the\DT@countiv} <
        \@nameuse{DT@level@\the\DT@countii}\relax
      \advance\dimen\z@ by-0.5\baselineskip
    \fi
    \kern-0.5\DT@rulewidth
    \hbox{\vbox to\z@{\vss\hrule width\DT@rulewidth height\dimen\z@}}%
    \kern-0.5\DT@rulewidth
    \kern-0.5\DT@dotwidth
    \vrule width\DT@dotwidth height0.5\DT@dotwidth depth0.5\DT@dotwidth
    \kern-0.5\DT@dotwidth
    \vrule width\DT@width height0.5\DT@rulewidth depth0.5\DT@rulewidth
    \kern\DT@sep
    \box\z@
    \endgraf
  \repeat
  \parindent=\DT@indent
  \parskip=\DT@parskip
  \baselineskip=\DT@baselineskip
  \let\strut\DT@strut
}
\def\DT@readarg.#1 #2. #3\@nil{%
  \advance\DT@counti \@ne
  \@namedef{DT@level@\the\DT@counti}{#1}%
  \@namedef{DT@body@\the\DT@counti}{\strut{\DTstyle{#2}\strut}}%
  \ifx\relax#3\relax
    \let\next\@gobble
  \fi
  \next#3\@nil
}
\catcode`\@=\DTAtCode\relax

\newtheorem{defn}{Definition}

\newtheorem{manualthm}{Theorem}
\setlength{\fboxsep}{1pt}
\setlength{\fboxrule}{1pt}
\newtheorem{innercustomthm}{Theorem}
\newenvironment{customthm}[1]
  {\renewcommand\theinnercustomthm{\fbox{#1}}\innercustomthm}
  {\endinnercustomthm}
  \newtheorem{innercustomconsthm}{Construction and Theorem}
\newenvironment{customconsthm}[1]
  {\renewcommand\theinnercustomconsthm{\fbox{#1}}\innercustomconsthm}
  {\endinnercustomconsthm}
  
\newtheorem{innercustomlem}{Lemma}
\newenvironment{customlem}[1]
  {\renewcommand\theinnercustomlem{\fbox{#1}}\innercustomlem}
  {\endinnercustomlem}
  
\newtheorem{innercustomcons}{Construction}
\newenvironment{customcons}[1]
  {\renewcommand\theinnercustomcons{\fbox{#1}}\innercustomcons}
  {\endinnercustomcons}
  \newtheorem{innercustomcor}{Corollary}
\newenvironment{customcor}[1]
  {\renewcommand\theinnercustomcor{\fbox{#1}}\innercustomcor}
  {\endinnercustomcor}
  
\newtheorem{innercustomdefn}{Definition}
\newenvironment{customdefn}[1]
  {\renewcommand\theinnercustomdefn{\fbox{#1}}\innercustomdefn}
  {\endinnercustomdefn}
  
\newtheorem{definitions}[defn]{Definitions}

\newtheorem{thm}{Theorem}
\newtheorem{lem}[thm]{Lemma}
\newtheorem{construction}[defn]{Construction}
\newtheorem{prop}[thm]{Proposition}
\newtheorem*{prop*}{Proposition}
\newtheorem{cor}[thm]{Corollary}
\newtheorem{conj}[thm]{Conjecture}

\newtheorem{ques}[thm]{Question}
\newtheorem{claim}[thm]{Claim}
\newtheorem*{claim*}{Claim}
\newtheorem{algo}[defn]{Algorithm}
\theoremstyle{remark}
\newtheorem{rem}[defn]{Remark}
\theoremstyle{remark}
\newtheorem{remarks}[defn]{Remarks}
\theoremstyle{remark}
\newtheorem{notation}[defn]{Notation}
\theoremstyle{remark}
\newtheorem{exmp}[defn]{Example}
\theoremstyle{remark}
\newtheorem{examples}[defn]{Examples}
\theoremstyle{remark}
\newtheorem{dgram}[defn]{Diagram}
\theoremstyle{remark}
\newtheorem{fact}[defn]{Fact}
\theoremstyle{remark}
\newtheorem{illust}[defn]{Illustration}
\theoremstyle{remark}
\newtheorem{que}[defn]{Question}
\numberwithin{equation}{section}
\newtheorem{example}[defn]{Example}
\newtheorem{exercise}[defn]{Exercise}

\def\abs#1{\lvert#1\rvert}
\def\MSB{MSB }
\def\pivot{\mathsf{pivot}}
\def\tred{\textcolor{red}}
\def\dim{\mathsf{dim}}
\def\MS{\mathsf{MS}}
\def\res{\upharpoonright}
\def\tensor{\bigotimes}
\def\defeq{\vcentcolon=}
\def\meet{\wedge}
\def\join{\vee}
\def\dLOfpb#1#2{\mathrm{dLO}_\mathrm{fp}^{{#1}{#2}}}
\def\corner{F}
\def\less{\prec}
\def\frame#1{\begin{mdframed}#1\end{mdframed}}
\def\hdammock{$\Tilde{h}$ammock }
\def\bb{\mathfrak{b}}
\renewcommand{\labelitemii}{$ \circ $}
\def\sd{sd}
\def\eqdef{=\vcentcolon}
\def\upset{\uparrow}
\def\downset{\downarrow}
\def\length{\mathsf{length}}
\def\gray{\textcolor{gray}}
\def\teal{\textcolor{teal}}
\def\lime{\textcolor{lime}}
\def\magenta{\textcolor{magenta}}
\def\orange{\textcolor{orange}}
\newcommand{\Il}{\pi_l(\ii)}
\def\tblue{\textcolor{blue}}
\newcommand{\Ir}{\pi_r(\ii)}
\newcommand{\htIl}{\htpi_l(\ii)}
\newcommand{\htIr}{\htpi_r(\ii)}
\newcommand{\STT}{\mathsf{long}}
\newcommand{\smol}{\mathsf{short}}
\newcommand{\beeg}{\mathsf{long}}
\newcommand{\eqvl}{\equiv_H^l}
\newcommand{\eqvr}{\equiv_H^r}
\newcommand{\eqvlr}{\equiv_H^{lr}}
\newcommand{\eqv}{\equiv_{H}}

\newcommand{\Lo}{\mathsf{L^0}}
\newcommand{\Lp}{\mathsf{L^+}}
\newcommand{\Lm}{\mathsf{L^-}}
\newcommand{\sfL}{\mathsf{L}}
\newcommand{\Ro}{\mathsf{R^0}}
\newcommand{\Rp}{\mathsf{R^+}}
\newcommand{\Rm}{\mathsf{R^-}}
\newcommand{\sfR}{\mathsf{R}}

\newcommand{\DL}{\Delta_\sfL}
\newcommand{\DR}{\Delta_\sfR}

\newcommand{\hb}{\widebar{H}}
\newcommand{\hht}{\widehat{H}}
\newcommand{\hd}{\widetilde H}
\newcommand{\Hb}{\overline{\mathcal H}}
\newcommand{\Hht}{\widehat{\mathcal H}}
\newcommand{\Hd}{\widetilde{\mathcal H}}
\newcommand{\infb}{\,^\infty\bb}
\newcommand{\binf}{\,\bb^\infty}
\newcommand{\qb}{\mathsf Q^\mathsf {Ba}}
\newcommand{\suc}{\mathfrak{succ}}
\newcommand{\pred}{\mathfrak{pred}}
\newcommand\A{\mathcal{A}}
\newcommand\B{\mathsf{B}}
\newcommand\BB{\mathcal{B}}
\newcommand\C{\mathcal{C}}
\newcommand\Pp{\mathcal{P}}
\newcommand\D{\mathcal{D}}
\newcommand\Hamm{\hat{H}}
\newcommand\hh{\mathfrak{h}}
\newcommand\HH{\mathcal{H}}
\newcommand\RR{\mathcal{R}}
\newcommand\Red[1]{\mathrm{R}_{#1}}
\newcommand\HRed[1]{\mathrm{HR}_{#1}}
\newcommand\K{\mathcal{K}}
\newcommand\LL{\mathcal{L}}
\newcommand\Lim{\text{{\normalfont Lim}}}
\newcommand\M{\mathcal{M}}
\newcommand\SD{\mathcal{SD}}
\newcommand\MD{\mathcal{MD}}
\newcommand\SMD{\mathcal{SMD}}
\newcommand\T{\mathcal{T}}
\newcommand\TT{\mathfrak T}
\newcommand\ii{\mathcal I}
\newcommand\UU{\mathcal{U}}
\newcommand\VV{\mathcal{V}}
\newcommand\ZZ{\mathcal{Z}}

\newcommand\Q{\mathbb{Q}}
\newcommand{\N}{\mathbb{N}} 
\newcommand{\R}{\mathbb{R}}
\newcommand{\Z}{\mathbb{Z}}
\newcommand{\qq}{\mathfrak q}
\newcommand{\ch}{\circ_H}
\newcommand{\cg}{\circ_G}
\newcommand{\bua}[1]{\mathfrak b^{\alpha}(#1)}
\newcommand{\falpha}{{\mathfrak{f}\alpha}}
\newcommand{\fgamma}{\gamma^{\mathfrak f}}
\newcommand{\fbeta}{{\mathfrak{f}\beta}}
\newcommand{\bub}[1]{\mathfrak b^{\beta}(#1)}
\newcommand{\bla}[1]{\mathfrak b_{\alpha}(#1)}
\newcommand{\blb}[1]{\mathfrak b_{\beta}(#1)}
\newcommand{\lmin}{\lambda^{\mathrm{min}}}
\newcommand{\lmax}{\lambda^{\mathrm{max}}}
\newcommand{\xmin}{\xi^{\mathrm{min}}}
\newcommand{\xmax}{\xi^{\mathrm{max}}}
\newcommand{\lbmin}{\bar\lambda^{\mathrm{min}}}
\newcommand{\lbmax}{\bar\lambda^{\mathrm{max}}}
\newcommand{\ff}{\mathfrak f}
\newcommand{\cc}{\mathfrak c}
\newcommand{\dd}{\mathfrak d}
\newcommand{\sqsf}{\sqsubset^\ff}
\newcommand{\rr}{\mathfrak r}
\newcommand{\pp}{\mathfrak p}
\newcommand{\uu}{\mathfrak u}
\newcommand{\vv}{\mathfrak v}
\newcommand{\ww}{\mathfrak w}
\newcommand{\xx}{\mathfrak x}
\newcommand{\yy}{\mathfrak y}
\newcommand{\zz}{\mathfrak z}
\newcommand{\MM}{\mathfrak M}
\newcommand{\mm}{\mathfrak m}
\newcommand{\sbq}{\mathfrak s}
\newcommand{\tbq}{\mathfrak t}
\newcommand{\Spec}{\mathbf{Spec}}
\newcommand{\Br}{\mathbf{Br}}
\newcommand{\sk}[1]{\{#1\}}
\newcommand{\Prime}{\mathbf{Pr}}
\newcommand{\Parent}{\mathbf{Parent}}
\newcommand{\Uncle}{\mathbf{Uncle}}
\newcommand{\Cousin}{\mathbf{Cousin}}
\def\sgn{\mathrm{sgn}}

\newcommand{\Nephew}{\mathbf{Nephew}}
\newcommand{\Sibling}{\mathbf{Sibling}}
\newcommand{\uc}{\mathrm{uc}}
\newcommand{\MCP}{\mathrm{MCP}}
\newcommand{\MSCP}{\mathrm{MSCP}}
\newcommand{\TTT}{\widetilde{\T}}
\newcommand{\la}{l}
\newcommand{\ra}{r}
\newcommand{\lb}{\bar{l}}
\newcommand{\rb}{\bar{r}}
\newcommand{\tBa}{\varepsilon^{\mathrm{Ba}}}
\def\ii{\mathcal{I}}
\newcommand{\brac}[2]{\langle #1,#2\rangle}
\newcommand{\braket}[3]{\langle #1\mid #2:#3\rangle}
\newcommand{\fin}{fin}
\newcommand{\inff}{inf}
\newcommand{\Zg}{\mathrm{Zg}(\Lambda)}
\newcommand{\Zgs}{\mathrm{Zg_{str}}(\Lambda)}
\newcommand{\STR}[1]{\mathrm{Str}(#1)}
\newcommand{\dmod}{\mbox{-}\operatorname{mod}}
\newcommand{\HQ}{\mathcal{HQ}^\mathrm{Ba}}
\newcommand{\bHQ}{\overline{\mathcal{HQ}}^\mathrm{Ba}}
\newcommand\Af{\mathcal{A}^{\ff}}
\newcommand\AAf{\bar{\mathcal{A}}^{\ff}}
\newcommand\Hf{\mathcal{H}^{\ff}}
\newcommand\Rf{\mathcal{R}^{\ff}}
\newcommand\Tf{\T^{\ff}}
\newcommand\Uf{\mathcal{U}^{\ff}}
\newcommand\Sf{\mathcal{S}^{\ff}}
\newcommand\Vf{\mathcal{V}^{\ff}}
\newcommand\Zf{\mathcal{Z}^{\ff}}
\newcommand\bVf{\overline{\mathcal{V}}^{\ff}}
\newcommand\bTf{\overline{\mathcal{T}}^{\ff}}
\newcommand{\fmin}{\xi^{\mathrm{fmin}}}
\newcommand{\fmax}{\xi^{\mathrm{fmax}}}
\newcommand{\xif}{\xi^\ff}
\newcommand{\mycomment}[1]{}
\newcommand{\LOfp}{\mathrm{LO}_{\mathrm{fp}}}

\newcommand\Tl{\mathbf{T}}
\newcommand\Tla{\mathbf{T}_{\la}}
\newcommand\Tlb{\mathbf{T}_{\lb}}
\newcommand\Ml{\mathbf{M}}
\newcommand\Mla{\mathbf{M}_{\la}}
\newcommand\Mlb{\mathbf{M}_{\lb}}
\newcommand\OT{\mathcal{O}}
\newcommand\LO{\mathbf{LO}}

\def\tblue{\textcolor{blue}}

\newcommand\rad{\mathrm{rad}_\Lambda} 

\newcommand{\fork}[1]{\mathrm{Str}_{\text{Fork}}^{\la}(#1)}
\newcommand{\BaB}{\mathsf{Ba(B)}}
\newcommand{\CycB}{\mathsf{Cyc(B)}}
\newcommand{\ExtB}{\mathsf{Ext(B)}}
\newcommand{\St}{\mathsf{St}}
\def\short{\mathsf{short}}
\newcommand{\StB}[1]{\mathsf{St}_{#1}(\sB)}
\newcommand{\STB}[1]{\mathsf{St}_{#1}(\xx_0,i;\sB)}
\newcommand{\BST}[1]{\mathsf{BSt}_{#1}(\xx_0,i;\sB)}
\newcommand{\CSt}[1]{\mathsf{CSt}_{#1}(\sB)}
\newcommand{\CST}[1]{\mathsf{CSt}_{#1}(\xx_0,i;\sB)}
\newcommand{\ASt}[1]{\mathsf{ASt}_{#1}(\sB)}

\newcommand{\OSt}[1]{\overline{\mathsf{St}}_{#1}(\sB)}
\newcommand{\OST}[1]{\overline{\mathsf{St}}_{#1}(\xx_0,i;\sB)}
\newcommand{\lB}{\ell_{\sB}}
\newcommand{\lbB}{\overline{\ell}_{\sB}}
\newcommand{\LB}{l_{\sB}}
\newcommand{\LbB}{\overline{l}_{\sB}}

\newcommand{\BalB}{\mathsf{Ba}_l(\sB)}
\newcommand{\BalbB}{\mathsf{Ba}_{\lb}(\sB)}
\newcommand{\BarB}{\mathsf{Ba}_r(\sB)}
\newcommand{\BarbB}{\mathsf{Ba}_{\bar{r}}(\sB)}

\newcommand{\sB}{\mathsf{B}}
\newcommand{\Str}[1]{\mathrm{Str}_{#1}(\xx_0,i;\sB)}
\newcommand{\Strd}[1]{\mathrm{Str}'_{#1}(\xx_0,i;\sB)}
\newcommand{\Strdd}[1]{\mathrm{Str}''_{#1}(\xx_0,i;\sB)}
\newcommand{\Cent}{\mathrm{Cent}(\xx_0,i;\sB)}
\newcommand{\Start}{\mathrm{Start}(\xx_0,i;\sB)}
\newcommand{\End}{\mathrm{End}(\xx_0,i;\sB)}
\newcommand{\braclB}{\brac{1}{\lB}}
\newcommand{\braclbB}{\brac{1}{\lbB}}
\newcommand{\QBa}{\Q^{\mathrm{Ba}}}
\newcommand{\Ba}{\mathcal Q_0^\mathrm{Ba}}
\newcommand{\Cyc}{\mathsf{Cyc}(\Lambda)}
\newcommand{\lazy}{1_{(v, i)}}
\newcommand{\dLOfd}{\mathrm{dLO_{fd}}}
\newcommand{\dLOfdb}[2]{\mathrm{dLO}_{\mathrm{fd}}^{{#1}{#2}}}

\newcommand{\LOfd}{\mathrm{LO}_{\mathrm{fd}}}
\newcommand{\Balr}{((\BalB\cap \BarB) \cup (\BalbB \cap \BarbB))}
\newcommand{\Hbar}{\overline{H}_l^i(\xx_0)}
\newcommand{\Hhat}{\widehat{H}_l^i(\xx_0)}
\newcommand{\HOST}[1]{\widehat{\overline{\mathsf{St}}}_{#1}(\xx_0,i;\sB)}
\newcommand{\HHlix}{\widehat{H}_l^i(\xx_0)}

\def\lan{\langle}
\def\ran{\rangle}
\newcommand{\s}[1]{{\mathsf{#1}}}
\newcommand{\f}[1]{\mathfrak {#1}}
\newcommand{\ol}[1]{\overline{#1}}

\newcommand{\gst}{\s{St}}
\newcommand{\stbar}{\ol\gst}
\mathchardef\mh="2D 
\newcommand{\allst}{(\le\omega)\mh\s{St}(\s B)}
\def\da{\Downarrow}
\def\ua{\Uparrow}
\def\dotplus{\cdot+}
\def\rad{\text{rad}}
\def\comp{\circ}
\def\aa{\mathfrak a}
\def\to{\rightarrow}
\def\HR{\textsf{HR}}
\def\parens#1{{(#1)}}
\def\mbf{\mathbf}
\newcommand{\rank}{\mathsf{rank}}
\newcommand{\mfa}{\mathfrak{a}}
\def\lin{\mathbf{L}}
\def\two{\mathbbm{2}}

\makeatletter
\let\save@mathaccent\mathaccent
\newcommand*\if@single[3]{%
  \setbox0\hbox{${\mathaccent"0362{#1}}^H$}%
  \setbox2\hbox{${\mathaccent"0362{\kern0pt#1}}^H$}%
  \ifdim\ht0=\ht2 #3\else #2\fi
  }
\newcommand*\rel@kern[1]{\kern#1\dimexpr\macc@kerna}
\newcommand*\widebar[1]{\@ifnextchar^{{\wide@bar{#1}{0}}}{\wide@bar{#1}{1}}}
\newcommand*\wide@bar[2]{\if@single{#1}{\wide@bar@{#1}{#2}{1}}{\wide@bar@{#1}{#2}{2}}}
\newcommand*\wide@bar@[3]{%
  \begingroup
  \def\mathaccent##1##2{%
    \let\mathaccent\save@mathaccent
    \if#32 \let\macc@nucleus\first@char \fi
    \setbox\z@\hbox{$\macc@style{\macc@nucleus}_{}$}%
    \setbox\tw@\hbox{$\macc@style{\macc@nucleus}{}_{}$}%
    \dimen@\wd\tw@
    \advance\dimen@-\wd\z@
    \divide\dimen@ 3
    \@tempdima\wd\tw@
    \advance\@tempdima-\scriptspace
    \divide\@tempdima 10
    \advance\dimen@-\@tempdima
    \ifdim\dimen@>\z@ \dimen@0pt\fi
    \rel@kern{0.6}\kern-\dimen@
    \if#31
      \overline{\rel@kern{-0.6}\kern\dimen@\macc@nucleus\rel@kern{0.4}\kern\dimen@}%
      \advance\dimen@0.4\dimexpr\macc@kerna
      \let\final@kern#2%
      \ifdim\dimen@<\z@ \let\final@kern1\fi
      \if\final@kern1 \kern-\dimen@\fi
    \else
      \overline{\rel@kern{-0.6}\kern\dimen@#1}%
    \fi
  }%
  \macc@depth\@ne
  \let\math@bgroup\@empty \let\math@egroup\macc@set@skewchar
  \mathsurround\z@ \frozen@everymath{\mathgroup\macc@group\relax}%
  \macc@set@skewchar\relax
  \let\mathaccentV\macc@nested@a
  \if#31
    \macc@nested@a\relax111{#1}%
  \else
    \def\gobble@till@marker##1\endmarker{}%
    \futurelet\first@char\gobble@till@marker#1\endmarker
    \ifcat\noexpand\first@char A\else
      \def\first@char{}%
    \fi
    \macc@nested@a\relax111{\first@char}%
  \fi
  \endgroup
}
\makeatother
\newcommand\test[1]{%
$#1{M}$ $#1{A}$ $#1{g}$ $#1{\beta}$ $#1{\mathcal A}^q$
$#1{AB}^\sigma$ $#1{H}^C$ $#1{\sin z}$ $#1{W}_n$}

\def\eps{\varepsilon}

\def\beps{\bm{\eps}}
\def\inorder{<_\two}
\def\emptyseq{\lan\ran}
\def\LOfa{\mathrm{LO_{fa}}}

\def\R{\mathbb{R}}
\def\diam{\diamondsuit}
\def\club{\clubsuit}
\def\P{\mathcal{P}}
\def\res{\upharpoonright}
\def\name{\mathring}
\def\A{\mathcal{A}}
\def\dom{\text{dom}}
\def\K{\mathcal{K}}
\def\Pbar{\overline{P}}
\def\bf{\mathbf}
\def\lessp{\leq^{pr}_}
\def\lessap{\leq^{apr}_}
\def\lessstar{\leq_{\star}}
\def\I{\bf{I}}

\def\vsk{\vspace{6pt}}
\def\cal{\mathcal}
\def\frak{\mathfrak}
\def\e{\varepsilon}
\def\club{\clubsuit}
\def\diam{\diamondsuit}
\def\Cohen{\textsf{Cohen}}
\def\Random{\textsf{Random}}
\def\cf{\textsf{cf}}
\def\dom{\textsf{dom}}
\def\rng{\textsf{range}}
\def\supp{\textsf{supp}}
\def\odim{\textsf{odim}}
\def\cont{\frak{c}}
\def\lt{\textsf{left}}
\def\rt{\textsf{right}}
\def\up{\textsf{up}}
\def\down{\textsf{down}}
\def\otp{\textsf{otp}}
\def\rk{\textrm{rk}}
\def\opt{\textsf{otp}}
\def\add{\textsf{add}}
\def\split{\textsf{Split}}
\def\Meager{\textsf{Meager}}
\def\Null{\textsf{Null}}
\mathchardef\mhyphen="2D
\def\blank{\mhyphen}
\def\deltahat{\hat{\delta}}
\def\call{\mathcal}
\def\red{\textcolor{red}}
\def\blue{\textcolor{blue}}
\def\LOfa{\mathrm{LO}_\mathrm{fa}}
\def\M{{M_{1b}}}
\def\m{{M_{1a}}}
\def\preorder{\text{$<_\textit{pre} $}}
\def\front{\mathrm{front}}
\def\Front{\mathbf{Front}}

\vspace{-1cm}
\def\st{\mathrm{st}}
\def\res{\upharpoonright}

\begin{abstract}
Given a special biserial algebra $\Lambda$ over an algebraically closed field, let $\mathrm{rad}_\Lambda$ denote the radical of its module category.
The authors showed with Sinha that the stable rank of a special biserial algebra $\Lambda$, i.e., the least ordinal $\gamma$ satisfying $\mathrm{rad}_\Lambda^\gamma=\mathrm{rad}_\Lambda^{\gamma+1}$, is strictly bounded above by $\omega^2$. 
We use finite automata to give simple algorithmic proofs, complete with their time complexity analyses, of two key ingredients in the proof of this result--the first one states that certain linear orders called hammocks associated with such algebras are \emph{finite description linear orders}, i.e., they lie in the smallest class of linear orders that contains finite linear orders and $\omega$, and that is closed under isomorphisms, order-reversals, binary sums, co-lexicographic products and finitary shuffles. We also document a complete proof of the result that the class of order types(=order-isomorphism classes) of finite description linear orders coincides with that of languages of finite automata under inorder.
\end{abstract}
\maketitle
\vspace{-6mm}
\section*{Introduction}
This paper contributes to the representation theory of special biserial algebras by using finite automata to compute their stable ranks, thus exploring a striking application of computer science apparatus in representation theory. The authors hope that this work, especially due to its expository aspects, will bridge the gap between these two fields and aid current and future representation theorists in applying these techniques in their work.
 
Fix an algebraically closed field $ \mathcal K $. The stable rank, $ \mathrm{st}(\Lambda) $, of a finite-dimensional associative $ \mathcal K $-algebra $ \Lambda $ is an ordinal-valued Morita invariant measuring the complexity of factorizations in the module category $\Lambda\text-\mathrm{mod}$ of finite length (equivalently, finite-dimensional) left modules over $\Lambda$. More specifically, it is the least ordinal at which the ordinal-indexed descending chain  \[\Lambda\text-\mathrm{mod}=\mathrm{rad}_\Lambda^0 \supseteq \mathrm{rad}_\Lambda^1 \supseteq \dots \supseteq \mathrm{rad}_\Lambda^\omega\supseteq \mathrm{rad}_\Lambda^{\omega+1} \supseteq\dots \] of transfinite-recursively defined ordinal powers of the radical $ \rad_\Lambda $ stabilizes, where $\mathrm{rad}_{\Lambda} $  is the two-sided ideal of the module category $ \Lambda\text-\mathrm{mod} $ generated by the non-invertible(=non-split) morphisms between indecomposables. The rank $ \rk(f) $ of any morphism $ f $ in $ \Lambda\text-\mathrm{mod} $ is defined to be the least ordinal $ \alpha $, if one exists, for which $ f \not\in \rad_\Lambda^{\alpha + 1} $; otherwise it is defined to be $\infty$.
 
A special biserial algebra over $ \mathcal K $ admits presentation as the quotient $ \mathcal K \mathcal Q / \lan \rho \ran $ of the path algebra generated by a finite connected quiver $ \mathcal Q $ by the two-sided ideal generated by a  finite set $ \rho $ of blocking relations, with the pair $ (\mathcal Q, \rho) $ satisfying certain local restrictions on the quiver (Definition \ref{defn: string algebra}). A strengthening of this definition, where $\rho$ consists only of paths, gives the definition of string algebras. The class of special biserial algebras and its subclass consisting of string algebras form testing grounds for various conjectures about the tame representation type due to the explicit combinatorial descriptions of their Auslander-Reiten quivers. The class of tame representation type algebras is further divided into the classes of domestic and non-domestic algebras respectively, each of which intersects both the classes of special biserial and string algebras.

The main result under consideration in this paper is the following.

\begin{customthm}{\textbf{Special}}\cite[Theorem~1]{SSK}
\label{thm: Special}
If $\Lambda$ is a special biserial $\mathcal K$-algebra, then $\st(\Lambda)<\omega^2$. In fact, there are natural numbers $n,d,e$ that depend of the bound quiver $(\mathcal Q,\rho)$ such that $0\leq\st(\Lambda)-(\omega\cdot n+d)\leq e$.
\end{customthm}

In his PhD thesis, Schr\"oer \cite{Schroer98hammocksfor} proved this result in the domestic case (also see \cite{schroer2000infinite}), and his approach was through the study of several interrelated partially ordered sets associated with each vertex $v$ of the quiver $Q$, called \emph{hammock posets}, denoted $\mathcal H(v) \supseteq\overline H(v)\supseteq H(v)$, and the two projections of $H(v)$, called \emph{hammock linear orders} and denoted $H_l(v)$ and $H_r(v)$, which encode factorizations of certain morphisms between indecomposables. Theorem \ref{thm: Special} was originally stated as \cite[Conjecture~4.4.1]{GKS20}, where it was proven for a subclass of non-domestic string algebras.

We need to establish some order-theoretic terminology before we can state a key ingredient in the proof of the result in general case. The class of order types(=order-isomorphism classes) of linear orders is equipped with some finitary operations, namely associative but non-commutative sum ($+$) and co-lexicographic product ($\cdot$), and, for each positive integer $n$, an $n$-ary shuffle operation ($\Xi$). The class of finite description linear orders is the smallest class of linear orders containing finite linear orders and $\omega$, the order type of the standard ordering on the natural numbers, and that is closed under isomorphisms, order-reversals as well as all these finitary operations. 
\begin{customthm}{$ {{\alpha}} $}
\label{thm: hammock computation}
\label{thm: alpha}\cite[Theorem~11.9]{SKSK}
    For any vertex $ v $ in a quiver presentation of a string algebra $ \Lambda $, the hammock linear orders $ H_l(v) $ and $ H_r(v) $ are finite description linear orders.
\end{customthm}
Recall that a poset is said to be \emph{scattered} if the ordered set of rationals cannot be embedded in it. In the domestic case, the above result was proven in \cite{SK}, where it was also shown that the hammocks are scattered.

All finite description linear orders have an associated isomorphism-invariant ordinal-valued \emph{density}, and the ordinal $\omega\cdot n+d$ can be computed using the densities of the hammock linear orders.

Now we discuss the error \emph{error term} $e$ in the statement of Theorem \ref{thm: Special}. A \emph{band point} is an element in $\overline H(v)\setminus H(v)$, and it is said to be  \emph{exceptional} if it lies in (the closure of) a scattered region in $\overline H(v)$. The error term $e$ is the total number of exceptional points in hammocks $\overline H(v)$ as $v$ varies over the set of vertices of $\mathcal Q$.

\begin{customthm}{E}\cite[Lemma~3.3.9]{SSK}
\label{thm: E}
 \label{thm: finitely many exceptional bands}
    For each vertex $v$, the number of exceptional points in $\overline H(v)$ is finite.
\end{customthm}

The above two results allowed the use of some of Schr\"oer's arguments to the non-domestic string algebra case. The proofs of the two ingredient results make use of heavy combinatorial arguments. The first major contribution of this paper is simpler, automata-theoretic proofs of these two results. We also discuss the time-complexities of the relevant algorithms (Remarks \ref{complexity alpha} and \ref{rem: E}).

The first, and perhaps the only, person to document the applicability of automata to the representation theory of string algebras was Rees \cite{rees2008automata}.

\parbox{0.85\textwidth}{
\epigraph{{\begin{center} It is well known that the sets of strings that define all representations of
string algebras and many representations of other quotients of path algebras form
a regular set, and hence are defined by finite state automata.\end{center}~}}{---\textup{Sarah Rees}, \cite{rees2008automata}}}

\begin{figure}[!htbp]
    \centering
    \begin{minipage}{.5\textwidth}
        \centering
\begin{tikzpicture}[scale=0.6][block/.style={rectangle, draw, minimum width=2cm, minimum height=1cm, text centered}]
 \draw[-] (-1, 0) -- (10, 0);
\draw[thin, fill=lightgray] (0,0) rectangle (1,4) node[midway] {\textbf{L}};
\draw[thin, fill=lightgray] (-0.5,4) rectangle (4.9,5) node[midway] {\textbf{String}};
\draw[thin, fill=lightgray] (3.4,2) rectangle (6.6,4) node[midway] {\textbf{M}} ;
\draw[thin, fill=lightgray] (3.6,1) rectangle (4.5, 2) node[midway] {\textbf{FT}};
\draw[thin, fill=lightgray] (3.3,0) rectangle (4.8,1) node[midway] {\tblue{$ \pmb{ \alpha} $}};
\draw[thin, fill=lightgray] (5.1,0) rectangle (6.5,2) node[midway] {\tblue{\textbf{E}}};
\draw[thin, fill=lightgray] (8,0) rectangle (9,5) node[midway] {\textbf{R}};
\draw[thin, fill=lightgray] (1.5,5) rectangle (10,6) node[midway] {\textbf{Special}};
\end{tikzpicture}
    \caption{Dependency of results related to the stable rank computation}
    \label{fig: big picture}
    \end{minipage}%
    \begin{minipage}{0.5\textwidth}
        \centering
    \begin{tikzcd}[scale=0.5]
	A && B \\
	\\
	D && C
	\arrow["\alpha", from=1-1, to=1-3]
	\arrow["\beta"', curve={height=6pt}, from=1-1, to=3-1]
	\arrow["{\beta'}", curve={height=-6pt}, from=1-1, to=3-1]
	\arrow["\epsilon"', from=1-3, to=3-1]
	\arrow["\gamma", from=3-1, to=3-3]
	\arrow["\delta", from=3-3, to=1-3]
\end{tikzcd}
    \caption{A quiver with relation $ \alpha = \delta\gamma\beta' $ depicting inter-relations between the objects of study}
    \label{fig: heart of intro}
    \end{minipage}%
\end{figure} 
\renewcommand\DTstyle{\rmfamily}

Figures \ref{fig: big picture} and \ref{fig: heart of intro} depict the relationship between different components of the story. While the interdependence of various results used in the proof of Theorem \ref{thm: Special} is depicted in Figure \ref{fig: big picture}, the quiver in Figure \ref{fig: heart of intro} along with the relation in its caption describes constructions and results relating the new ideas involving automata theory as well as the organisation of sections of the paper.

Here is a brief description of various components of Figure \ref{fig: big picture}.
\dirtree{%
.1 \textbf{Special} denotes Theorem \ref{thm: Special}, i.e., the stable rank computation for special biserial algebras {.}.
.2 \textbf{R} is the connection between the representation theories of special biserial and string algebras \cite{skowronski1983representation, wald1985tame}{.}.
.2 \textbf{String} denotes the problem restricted to string algebras \cite[\S~3.4]{SSK}{.}.
.3 \textbf{L} represents the connection from \cite[{Pg.~66}]{Schroer98hammocksfor} between the rank of a graph map and the density of the associated interval in a hammock poset{.}.
.3 \textbf{M} is density computation of scattered intervals of hammock posets \cite[Theorem~3.3.3(4)]{SSK}{.}.
.4 \textbf{FT} is the result that hammock posets are of \emph{finite type}, i.e., have finitely many maximal \hphantom{1pt}scattered regions up to order-isomorphism{.} \cite[Theorem~3.3.3(2)]{SSK}.
.5 \textbf{\tblue{$ \pmb{\alpha} $}} is Theorem \ref{thm: alpha} \cite[Theorem~11.9]{SKSK}{.}.
.4 \textbf{\tblue{E}} is Theorem \ref{thm: E} \cite[Lemma~3.3.9]{SSK}{.}.
}

Let $ \Lambda = \mathcal K \mathcal Q/ \lan \rho \ran $ be a string algebra for the rest of the introduction. Rees \cite{rees2008automata} showed that the set of certain combinatorial entities known as \emph{strings}, that parametrizes certain indecomposable modules for $\Lambda$, is computable using a finite automaton over an alphabet comprising of the arrows of the quiver $ \mathcal Q$ along with their formal inverses (Theorem~\ref{thm: beta}). However, that paper did not mention hammocks or order relations between strings. This paper builds upon Rees' approach by showing that a modified automaton  captures hammock order relations between strings. 
We construct, for a vertex $ v $ in $ \mathcal Q $, an automaton (Construction~\ref{cons: beta'}) over the alphabet $ \two := \{0, 1\} $ which captures the hammock linear order $(H_l(v),<_l)$ (Theorem \ref{thm: beta'}).
This shows that hammock linear orders are computable using finite automata in the sense that they are order-isomorphic to \emph{linguages}(=languages under \emph{inorder}) of automata over the alphabet $ \two $. We then invoke a result due to Heilbr\"unner \cite{heilbrunner1980algorithm} asserting that any such linear order admits a finite description, thereby bypassing all the combinatorics in \cite{SKSK}.

It is indeed remarkable that the order types of finite description linear orders, whose class was first studied by L\"auchli and Leonard \cite{leonard1968elementary} while studying the model theory of linear orders, are exactly the order types of the linguages of finite automata (Theorem \ref{thm: automata and linear orders}). The second major contribution of the paper is to make the underlying automata theory more accessible to representation theorists, because it is our view that much potential exchange between these subjects is prevented by the very different languages and terminologies employed in the technical literature of these areas.

Now we discuss the organization of the rest of the paper through a brief description of vertices as well as arrows of the quiver in Figure~\ref{fig: heart of intro}. 

\noindent\textbf{$ A : $ String algebras:} In this section we recall the definition of string algebras and some basic results from their representation theory, including the definition of hammock posets.

\noindent\textbf{$ B : $ Finite description linear orders:} After recalling some preliminaries on linear orders and basic operations on them, we define the class of finite description linear orders.

\noindent\textbf{$ C : $ Word problems:} Word problems are essentially context-free grammars with a single production rule for every non-terminal symbol.
We list important results due to Courcelle \cite{courcelle1978frontiers} and Heilbr\"unner \cite{heilbrunner1973gleichungssysteme, heilbrunner1980algorithm} on the solutions to such problems. We also recall Courcelle's construction of a family of trees whose ``frontiers'' form the universal solutions to a given word problem (Construction \ref{cons: delta}, Theorem \ref{thm: delta}).

\noindent\textbf{$ D : $ Finite automata:} We start by recalling the basics of the theory of deterministic finite automata and then state  Rees' theorem  (Theorem~\ref{thm: beta}) on the existence of a finite automaton accepting only the strings for a string algebra.
We then describe a one-to-one correspondence  between order types of finite description linear orders and linguages of finite automata over a totally ordered alphabet (Theorem~\ref{thm: automata and linear orders}).

To prove the easier direction, we describe the construction of a finite automaton with linguage isomorphic to a given finite description linear order (Construction and Theorem~\ref{consthm: epsilon}).

For the converse, we start by proving that the linguage of any finite automaton arises as the linear order underlying the universal solution to a word problem (Construction \ref{cons: gamma}, Theorem \ref{cons: gamma}). The result then follows from Theorem \ref{thm: delta}.

\noindent\textbf{$ \beta' : $ Hammocks as linguages:} In this section, we construct, for a fixed vertex $ v $ in a string algebra $ \Lambda $, an automaton $ M'_\Lambda(v) $ with linguage isomorphic to $ H_l(v) $ (Construction \ref{cons: beta'}, Theorem \ref{thm: beta'}), which together with Theorems \ref{thm: gamma} and \ref{thm: delta} gives the promised automata-theoretic proof of Theorem \ref{thm: alpha}.

\noindent\textbf{$ E : $ Exceptional bands:} In this section, we give the promised automata-theoretic proof of Theorem~\ref{thm: E}, obtaining the final ingredient for stable rank computation \`a la \cite{SSK}.

Here are some notations and conventions. The set of natural numbers is denoted $\mathbb N$ and contains $0$. We identify any natural $ n $ with the set $ \{0, 1, \dots, n - 1\}$ of smaller naturals. Similarly, the set of rational numbers is denoted $ \mathbb Q $. We will often confuse a linear (resp. partial) order with its order type. For example, $ \mbf n $, $ \omega $, $ \omega^* $ and  $ \eta $ denote the order types of $ \{0, 1, \dots, n - 1\} $, the non-negative integers, the non-positive integers and the rationals under their usual order relations respectively.

\section*{$ A : $ String algebras}
\label{section: A}

In this section, we recall just as much of the theory of string algebras as is needed for the rest of the paper. For more detailed accounts, the interested reader can refer to \cite{assem2006elements} 
for finite-dimensional algebras/quiver representations and to \cite[\S~3]{SSK} for string algebras.

\begin{defn}
\label{defn: string algebra}

An algebra $ \Lambda $ presented as the quotient $ \mathcal K \mathcal Q/ \lan \rho \ran $ of the path algebra $ \mathcal K \mathcal Q $ generated by the quiver $ \mathcal Q = (Q_0, Q_1, s, t) $, having \begin{itemize}
\item $ Q_0 $ as its set of vertices,
\item $ Q_1 $ as its set of arrows, and
\item $ s, t : Q_1 \to Q_0 $ as the source and target maps respectively,
\end{itemize}
by its admissible ideal $ \lan \rho \ran $ generated by a finite set $ \rho $ of blocking relations given by paths in the quiver is said to be a \emph{string algebra} if the following conditions hold for every $ v \in Q_0 $ and $ b \in Q_1 $:

\begin{itemize}
    \item there are at most two arrows with source $ v $ and at most two arrows with target $ v $,
    \item there is at most one arrow $ a $ with $ ab \not\in \rho $ and at most one arrow $ c $ with $ bc \not\in \rho $.
\end{itemize}
\end{defn}
\begin{rem}
If we drop the condition that $ \rho $ consists of paths in the above definition, then we obtain the definition of a \emph{special biserial algebra}.
\end{rem}

For technical reasons we also choose and fix auxiliary maps $ \sigma, \tau : Q_1 \to \{-1, 1\} $ satisfying the following conditions for any two distinct syllables $\alpha, \beta \in Q_1 $:
\begin{enumerate}
\item  if $ s (\alpha ) = s(\beta) $, then $ \sigma (\alpha) = -\sigma (\beta) $;

\item  if $ t (\alpha) =t(\beta) $, then $ \tau(\alpha) = -\tau (\beta) $; and
\item If $ s(\alpha) = t(\beta) $ and $ \alpha\beta \not\in \rho $, then $ \sigma(\alpha) = -\tau(\beta) $.
\end{enumerate}
\begin{example}
\label{example: string algebra}
The string algebra $ GP_{2, 3} $ presented by the quiver and relations in Figure~\ref{fig: string algebra} was first studied by Gel'fand and Ponomarev. A choice of $ \sigma $ and $ \tau $ maps for this algebra is
$ \sigma (a) =\tau (b) = -1 $, $ \sigma (b) = \tau (a) = 1 $.
    \begin{figure}[htbp]
    \centering
\begin{tikzcd}
	v
	\arrow["a", from=1-1, to=1-1, loop, in=145, out=215, distance=10mm]
	\arrow["b", from=1-1, to=1-1, loop, in=325, out=35, distance=10mm]
\end{tikzcd}
\caption{The quiver for the string algebra $ GP_{2, 3} $, for which $ \rho = \{a^2, b^3, ab, ba\} $}
    \label{fig: string algebra}
\end{figure}
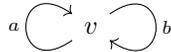
\end{example}
The arrows in $ Q_1 $ are called \emph{direct syllables} and are labelled by lowercase Roman letters, possibly with decoration. 
To each direct syllable $ \alpha $, we associate a formal \emph{inverse syllable} $ \alpha^- $, for which $ s(\alpha^-) := t(\alpha) $ and $ t(\alpha^-) := s(\alpha) $. If $ \alpha $ is the inverse of a direct syllable $ \beta $, then we define $ \alpha^- := \beta $. Inverse syllables form the set $ Q_1^- $. 

We now define combinatorial entities called \emph{strings} and \emph{bands} which yield much information about the representations of the string algebra $ \Lambda $. 

\begin{defn}[Strings and bands]
    A walk is either a zero-length walk of the form $1_{(v,i)}$ for $v\in Q_0, i\in\{1,-1\}$ or a word $ \xx = \alpha_1 \alpha_2 \dots \alpha_n $ over the alphabet $ Q_1 \sqcup Q_1^- $, $ n \ge 1$, such that for every $ 0 < i < n $, we have $ s(\alpha_i) = t(\alpha_{i + 1}) $. (Thus walks are read from right to left).
    
    We denote by $ \xx^- $ the \emph{inverse walk} of $\xx$ that has the form $1_{(v,-i)}$ or $ \alpha_n^- \dots \alpha_2^- \alpha_1^- $ respectively.

    A walk $ \xx $ is said to be a \emph{string} if 
    \begin{itemize}
        \item there do not exist $ 1 \le i < j \le n $ with either $ \alpha_i \alpha_{i+1} \dots \alpha_j $ or its inverse in $ \rho $, and
        \item there does not exist $ 1 \le i \le n $ with $ \alpha_i = \alpha_{i+1}^- $.
    \end{itemize} 
    
A zero-length walk is always a string. The set of all strings for the algebra $ \Lambda $ is denoted $ \St(\Lambda) $.

The source and target functions can be easily extended to all strings by $s(1_{(v,i)}):=v$ for $i\in\{1,-1\}$, $s(\alpha^-):=t(\alpha)$ for $\alpha\in Q_1$, $s(\alpha_1 \alpha_2 \dots \alpha_n):=s(\alpha_n)$ if $n\geq1$, and dually for the target function.

The $ \sigma $ and $ \tau $ functions can be extended by defining 
\begin{itemize}
    \item $ \sigma (1_{(v,i)}) = -i $
and $ \tau(1_{(v,i)}) = i $ for $ v \in Q_0 $ and $ i \in \{1, -1\} $,
\item $ \sigma(\alpha^-) :=\tau(\alpha) $, $ \tau(\alpha^-): = \sigma(\alpha) $ for each $ \alpha \in Q_1 $,
\item  $ \sigma(\alpha_1 \alpha_2 \dots \alpha_n) := \sigma(\alpha_n) $ and $ \tau(\alpha_1 \alpha_2 \dots \alpha_n) := \tau(\alpha_1) $ if $n\geq 1$.
\end{itemize}
For any string $ \xx $ and any $ (v, i) \in Q_0 \times \{1, -1\} $, the concatenation $ 1_{(v, i)}\xx $ (resp. $ \xx1_{(v, i)} $) is defined if $ t(\xx) = v $ and $ \tau(\xx) =i $ (resp. $ s(\xx) = v $ and  $ \sigma(\xx) =-i $).
    
A string $ \bb $ is said to be a \emph{band} if $ \bb^2 $ is a string, there is no substring $ \xx $ of $ \bb $ with $ \bb = \xx^n $ for some $ n > 1 $, and $ \bb $ starts with a direct syllable and ends with an inverse syllable.
\end{defn}
The next theorem, essentially due to Gel'fand and Ponomarev \cite{G&P}, characterizes all finite-dimensional indecomposable $ \Lambda $-modules.
\begin{thm}
    For any string $ \xx $, there is an associated indecomposable string module $ M(\xx) $, and for any band $ \bb $, positive integer $ n $ and non-zero scalar $ \lambda \in \mathcal K^* $, there is an associated indecomposable band module $ M(\bb, n, \lambda) $. A string module is not isomorphic to any band module. Moreover, for any strings $\xx$ and $\yy$, we have $M(\xx)\cong M(\yy)$ if and only if $\xx\in\{\yy,\yy^-\}$. Similarly, $M(\bb,n,\lambda)\cong M(\bb',n',\lambda')$ if and only if $\bb'$ is a cyclic permutation of $\bb$ or $\bb^-$, $n=n'$ and $\lambda=\lambda'$.
    Every finite-dimensional indecomposable $\Lambda$-module is isomorphic to either a string module or a band module.
\end{thm}
For any vertex $ v $ of the quiver $\mathcal Q$, let $ S $ be the associated simple module, and let $ P(S), I(S) $ be the projective cover and injective envelope of $ S $ respectively. Let $ f_S $ denote the composition $ P(S) \twoheadrightarrow  S \hookrightarrow I(S) $. Every non-zero map  $ g : M \to N $ between indecomposables whose image has a composition factor isomorphic to $S$, in the sense of the Jordan-H\"older theorem, is a factor of $f_S$, i.e., there are $h_1,h_2$ such that $f_S=h_2 g h_1$.

\begin{defn}
The \emph{hammock} $ \mathcal{H}(v) $ is the set of (isomorphism classes of) triples $ (N, g, h) $, where $  P(S) \xrightarrow[]{g} N \xrightarrow[]{h} I(S) $ is a factorization of $ f_S $ through an indecomposable module $ N $. The order $ < $ on $ \mathcal{H}(v) $ is defined by $ (N, g, h) < (N', g', h') $ if and only if $ g' $ factors through $ g $.
\end{defn}

The hammock $ \mathcal H(v) $ consists of string modules as well as band modules.
Denote by $H(v)$ its subposet consisting of string modules. The element $M(\xx)\cong M(\xx^-)$ of $H(v)$ can be thought of as the pair $(\xx_1,\xx_2)$ of strings, where $\xx=\xx_1\xx_2$ and $t(\xx_1)=s(\xx_2)=v$. Without loss of generality, we assume that $\tau(\xx_2)=1$.

The left (resp. right) projection of the poset $H(v)$ is a linear order $(H_l(v),<_l)$ (resp. $(H_r(v),<_r)$), where for different strings $\xx_1,\xx_2$, we have $\xx_1<_l\xx_2$ if and only if for the longest string $ \ww $ with $ \xx = \xx'\ww $ and $ \yy = \yy'\ww $ for some $ \xx', \yy' $, either $ \xx' $ has a direct rightmost syllable or $ \yy' $ has an inverse rightmost syllable. The ordering  $ <_r $ on $H_r(v)$ is defined by $\xx'_1 <_r\xx'_2$ if and only if ${\xx'_1}^-<_l{\xx'_2}^-$. Thus, we have $ H(v) = \{(\xx_1, \xx_2) \in H_l(v) \times H_r(v) \mid \xx_1\xx_2 \in \St(\Lambda)\} $.

In the poset $ \mathcal H(v) $, all the band modules of the form $ M(\bb, n ,\lambda) $ corresponding to a fixed band $ \bb $ occupy the same position with respect to $H(v)$--identifying such modules in $\mathcal H(v)$ for each band yields us the poset $ \overline H (v) $ containing $ H(v) $ as a subposet and having a single \emph{band point}, denoted formally as $ (^\infty\bb', {\bb'}^\infty) $ for each $ \bb' \in H_l(v) \cap H_r(v) $ that is a cyclic permutation of a band $\bb$. The linear orders $ \overline H_l(v) , \overline H_r(v)$ together with their order relations $ <_l $, $ <_r $ are the left and right projections of $ \overline H(v) $ respectively, while ensuring that the natural projection maps preserve order. 

Schr\"oer established a connection between the rank of a morphism between indecomposable modules for a string algebra and an order-theoretic invariant, called the \emph{density}, of the corresponding interval in the hammock posets $\overline H(v)$.

\begin{customthm}{$ \textbf{L} $}\cite[pp.66-67]{Schroer98hammocksfor}(cf. \cite[Theorem~3.4.4(4)]{SSK}) Suppose $ (\xx_1, \xx_2) < (\yy_1, \yy_2) $ in $ \overline H(v) $. Then the rank of the canonical map $ M(\xx_1\xx_2) \to M(\yy_1\yy_2) $ in $\mathcal H(v)$ depends only on the order type of the interval $ [(\xx_1, \xx_2), (\yy_1, \yy_2)] \subseteq \overline H(v) $.
\end{customthm}

\section*{$ B : $ Finite description linear orders}
\label{section: B}

In this section, we will define the class of finite description linear orders after introducing some key operations on linear orders. We will also recall trees and describe the frontier construction of a linear order from any tree. All the terminology which is not explained here can be found in the standard text \cite{rosenstein}.

The most common operations on linear orders are addition, multiplication and finitary shuffles, all of which are special instances of the  \emph{ordered sum} operation.

\begin{defn}
Given a linear order $ (L, <) $ and a sequence $ \lan (L_a, <_a) \mid a \in L \ran $ of linear orders, the \emph{ordered sum} $ \sum_{a \in L} L_a $ is defined to be the linear order formed by the set $ \{(a, x) \mid x \in L_a\} $ under the unique order relation $ < $ satisfying $ (a, x) < (a', x') \iff a < a' $ or ($ a = a' $ and $ x <_a x' $).
  
The \emph{sum} $ L_0 + L_1 $ and the \emph{product} $ L_0 \cdot L_1 $ are defined as $ \sum_{i \in \mbf 2} L_i $ and  $ \sum_{x \in L_1} L_0 $ respectively.
\end{defn}

To define finitary shuffles, we need Cantor's theorem on the existence and uniqueness of certain \emph{homogeneous colourings} of $ \mathbb{Q} $.
\begin{defn}
  For any natural $ n $, say that a coloring $ \chi : \mathbb Q \to n $ of the rational number line by $ n $ colors is \emph{homogeneous} if for any $ k < n $ the preimage $ \chi^{-1}(k) $ is dense in $ \mathbb Q $.
\end{defn}

\begin{thm}[\cite{rosenstein}, Theorems 7.11 and 7.13]
  For any natural $ n $, there exists a homogeneous $ n $-coloring of $ \mathbb{Q} $. Furthermore, any two such colourings $ \chi, \chi' $ are isomorphic in the sense that there exists an order automorphism $ f : \mathbb{Q} \to \mathbb{Q} $ with $ \chi \circ f = \chi' $.
\end{thm}

The $ n $-shuffle of $ L_0, L_1, \dots, L_{n-1} $ is obtained by starting with a homogeneous $ n $-coloring of $ \mathbb{Q} $ and then replacing every rational in $ \chi^{-1}(k) $ by a copy of $ L_k $ for each $ k < n $.

\begin{defn}
  Let $ L_0, L_1, \dots, L_{n - 1} $ be $ n $ linear orders. Then their \emph{$ n $-ary shuffle} $ \Xi(L_0, L_1, \dots, L_{n - 1}) $ is defined to be $ \sum_{r \in \mathbb{Q}} L_{\chi(r)} $, where $ \chi : \mathbb{Q} \to n $ is a homogeneous $ n $-coloring.
\end{defn}

\begin{defn}
\label{defn: LOfd}
  The class of \emph{finite description linear orders} is the smallest class of linear orders containing $ \mbf 0 $ and $ \mbf 1 $, and closed under additions, multiplication by $ \omega $, multiplication by $ \omega^* $, order isomorphisms and finitary shuffles.
\end{defn}

Linear orders often arise as frontiers of trees. 

\begin{defn}
Let $ (\omega^{< \omega}, \subseteq) $ denote the set of all finite sequences of naturals, where for $ x, y \in \omega^{< \omega} $ the notation $ x \subseteq y $  (read $ y $ extends $ x $) denotes that $ x $ is a prefix of $ y $.

    A \emph{tree} $ T $ is a subset of $ \omega^{< \omega} $ that is closed under prefixes. Elements of a tree are called \emph{nodes}. Every non-empty tree contains the sequence $ \lan \ran $, which we call the \emph{root node}. Nodes which are not proper prefixes of any other nodes are called \emph{leaf nodes}.

Given a set $D$, a \emph{$ D $-labelled tree} is a pair $ (T, \chi: T \to D)$ for a function $\chi$.
\end{defn}

There is a natural linear order on $ \omega^{< \omega} $, and hence on the leaf nodes of any tree. 

\begin{defn}
\label{defn: frontier}
    The \emph{preorder relation} $ \preorder $ on $ \omega^{< \omega} $ is defined by putting $ x $ $ \preorder $ $ y $ if either $ x \subsetneq y $ or $ x(n) < y(n) $ for the smallest natural $ n $ with $ x(n) \not= y(n) $.
    Given a tree $ T $, the set of its leaf nodes equipped with the preorder relation is called its \emph{frontier} and is denoted $ \mathbf{Front}(T) $.
\end{defn}

\section*{$ C : $ Word problems}
\label{section: C}
For this section, fix a finite set $ \Gamma $, called the alphabet.

\begin{defn}
    A word over $ \Gamma $ is a finite sequence of elements of $ \Gamma $.
    Given words $ x $ and $ y $, let $ xy $ denote their \emph{concatenation}, i.e., the word $ x $ followed by the word $ y $.
    Denote by $ \Gamma^* $ the set of all words over $ \Gamma $. The empty word is denoted $ \beps $.

  A \emph{word problem} over $ \Gamma $ is a system of equations
$ \lan v_i = t_i \mid i = 1, 2, \dots, k \ran $
  in which the distinct unknowns $ v_1, v_2, \dots, v_k $ are equated to words $ t_1, t_2, \dots, t_k $ respectively over the alphabet $ \Gamma \sqcup \{v_1, v_2, \dots, v_k\} $.
  A tuple $ \lan x_1, x_2, \dots, x_k \ran $ of words over the alphabet $ \Gamma $ is said to be a \emph{solution} of the word problem if the equations simultaneously hold true when we substitute each $ v_i $ by $ x_i $.
\end{defn}

Courcelle \cite{courcelle1978frontiers} and Heilbr\"unner \cite{heilbrunner1973gleichungssysteme, heilbrunner1980algorithm} contributed extensively to the study of word problems. 

\begin{example}
\label{example: word problem}
Consider the word problem $u={\star} u{\star}$ over the alphabet $ \Gamma = \{{\star}\} $. This word problem has no solution in finite words.
\end{example}

Courcelle generalized the notion of word to arbitrary length words by introducing arrangements, to solve all word problems.

\begin{defn}
  An \emph{arrangement} $ A $ over $ \Gamma $ is a pair $ (L, \chi) $, where $ L $ is a linear order and $ \chi: L \to \Gamma $ is a function, which we call the \emph{colouring function}. The \emph{concatenation of arrangements} $ (L, \chi) $ and $ (L', \chi') $ is defined to be $ (L + L', \chi \sqcup \chi') $.
\end{defn}

In \cite[Proposition~2.4, \S~4]{courcelle1978frontiers}, he showed that every word problem has a solution that is universal in a suitably-defined category of all solutions. In fact, given a word problem over $\Gamma$ in unknowns $v_1,\dots,v_n$, he showed that the universal solution arose as the frontier of an infinite tree obtained naturally from the word problem. For brevity, let $ \overline \Gamma $ denote $ \Gamma \sqcup \{v_1, v_2, \dots, v_n\} $.

\begin{customcons}{$ \delta $}
\label{cons: trees for word problems}
\label{cons: delta}
    Consider the word problem $ \lan v_i = x_{i1} x_{i2} \dots x_{ij_i} \mid 1 \le i \le n\ran $ over $\Gamma$ with $ j_i\geq1$ and $ x_{ik} \in\overline\Gamma$ for $ 1 \le k \le j_i $ for each $ i $. The \emph{associated infinite labelled trees} $ \lan (T_i\subseteq\omega^{<\omega},f_i:T_i\to\overline\Gamma)\mid 1\leq i\leq n\ran$ are obtained as the directed unions of $ \subseteq $-increasing sequences $ \lan (T_i^n,f_i^n:T_i^n \to\overline \Gamma) \mid n \in \N \ran $ of labelled finite trees satisfying for each $i$:
    \begin{enumerate}
        \item $ (T_i^0, \chi_i^0) $ is the labelled tree comprising of a root labelled $ v_i $ with $ j_i $ children, labelled $ x_{i1}, x_{i2}, \dots, x_{ij_i} $ from left to right;
        \item $ T_i^{n+1} $ is obtained from $ T_i^n $ by replacing each leaf node labelled $ v_j $ by a copy of $ T_j^0 $.
    \end{enumerate}
\end{customcons}
\begin{thm}[\cite{courcelle1978frontiers}]
\label{thm: courcelle frontier}
    The tuple  $ \lan (\Front(T_i), f_i\!\res_{\Front(T_i)}) \mid 1 \le i \le n \ran $ of arrangements is the universal solution to the word problem posed in Construction~\ref{cons: delta}.
\end{thm}

\begin{example}
\label{example: tree expansion}
For the word problem in Example~\ref{example: word problem}, Construction~\ref{cons: delta} gives the tree in Figure~\ref{fig: tree expansion} whose frontier is isomorphic to $ \omega + \omega^* $, which is the linear order underlying the universal solution $ (\omega + \omega^*, \chi: \omega + \omega^* \to \{{\star}\}) $. 
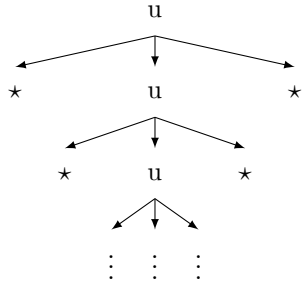
\begin{figure}[htbp]
    \centering
       
\begin{forest}
    [u [${\star}$\vspace{2cm}] [u [${\star}$\vspace{1cm}] [u [$ \vdots $] [$ \vdots $] [$ \vdots $]] [\vspace{1cm}${\star}$]] [${\star}$\vspace{2cm}]]]
\end{forest}
\caption{Tree expansion for $ u = {\star} u {\star}$}
    \label{fig: tree expansion}
\end{figure} 
\end{example}

Word problems were solved in full generality by Heilbr\"unner \cite{heilbrunner1980algorithm} based on the work in his PhD thesis \cite{heilbrunner1973gleichungssysteme}, who gave an algorithm to compute the universal solution using the basic operations on linear orders described in the previous section.

\begin{customthm}{$ \delta $}[\cite{heilbrunner1980algorithm}]
\label{thm: delta}
If the tuple $ \lan (L_i , \chi_i) \mid 1\leq i\leq n \ran $ of arrangements is the universal solution to a given word problem, then $ L_i $ is a finite description linear order for each $ i $.
\end{customthm}
\begin{rem}
\label{rem: delta}
    Heilbr\"unner's algorithm used in the proof of the above result takes a word problem as input and computes its initial solution explicitly in time linear in the input size.
\end{rem}

Courcelle identified those word problems for which the underlying linear orders of the arrangements in the universal solution are scattered.
\begin{defn}\cite[\S~3.5]{courcelle1978frontiers}
Using the notations of Construction \ref{cons: delta}, the word problem described there is said to be \emph{quasi-rational} if for every $ i $ for which $ T_i $ has a leaf node with a label in $ \Gamma $, of any two nodes of $ T_i $ labelled $ v_i $, one extends the other.
\end{defn}

\begin{thm}\cite[Theorem~3.8]{courcelle1978frontiers}
\label{prop: quasi-rational systems}
    A word problem is {quasi-rational} if and only if the linear orders underlying its universal solution are scattered.
\end{thm}
 
\section*{$ D : $ Automata}
In this section, we start by recalling deterministic finite automata, thus completing the descriptions of the vertices in Figure~\ref{fig: heart of intro}. After introducing \emph{linguages} (Definition~\ref{defn: linguages}), we describe arrows $ \beta $ (Theorem~\ref{thm: beta}), $ \epsilon $ (Construction and Theorem~\ref{consthm: epsilon})   and  $ \gamma $ (Construction~\ref{cons: gamma}, Theorem~\ref{thm: gamma}) from the same figure.

\label{section: D}
The reader interested in finite automata is referred to \cite{kozen2012automata} for a more comprehensive treatment.  

\begin{defn}
\label{defn: automata}
    Given a finite alphabet $ \Sigma $, a  \emph{(deterministic finite) automaton} $ M $ over $ \Sigma $ is a tuple $(Q, q_s, F, \delta) $, where $ Q $ is a finite set of \emph{states}, $ q_s \in Q $ is the \emph{start state}, $ \delta : (Q \times \Sigma) \to Q $ is a partial function governing \emph{state transition} and $ F \subseteq Q $ is the set of \emph{accept states}.
\end{defn}

Very often, we will consider automata over the alphabet $ \two := \{0, 1\} $.

Automata are finite state machines which either accept or reject their input--a word over $ \Sigma $.  Automata are easily described using finite labelled multigraphs (with loops allowed) with one additional sourceless unlabelled arrow, whose vertices denote the states, labelled arrows denote the transitions, the target of the unlabelled arrow marks the start state, and accept states are denoted with double boundaries.

\begin{example}
\label{example: finite automata}
Consider the automata $ M_1 $ and $ M_2 $ in Figure \ref{fig: M} over $ \two $, each with only one accept state.
    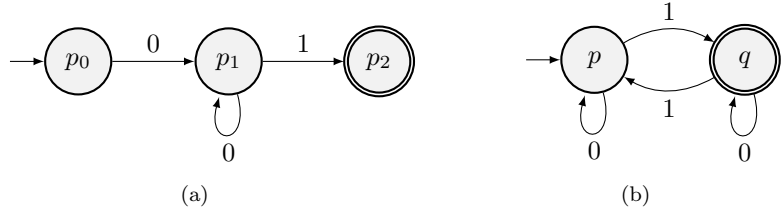
\begin{figure}[htbp]
        \centering
        \subfigure[]{\begin{tikzpicture}        
\node[state, initial left](root) {$ p_0 $};
\node[state] at (2, 0) (first) {$ p_1 $};
\node[state, accepting] at (4, 0) (second) {$ p_2 $};
\draw (root) edge[above] node{0} (first);
\draw (first) edge[above] node{1} (second);
\draw (first) edge[loop below, below] node{0} (first);   \end{tikzpicture}}\hspace{1cm}
        \subfigure[]{\begin{tikzpicture}        
\node[state, initial left](root) {$ p $};
\node[state, accepting] at (2, 0) (rt) {$ q $};
\draw (root) edge[bend left, above] node{1} (rt);
\draw (rt) edge[bend left, below] node{1} (root);
\draw (root) edge[loop below] node{0} (root);
\draw (rt) edge[loop below] node{0} (rt);       \end{tikzpicture}}
        \caption{(a) Automaton $ M_{1} $, (b) Automaton $ M_{2} $}
        \label{fig: M}
    \end{figure}

For $ M_1 $, we have $ \delta(p_0, 0) = \delta(p_1, 0) = p_1 $ and $ \delta(p_1, 1) = p_2 $.

For $ M_2 $, we have $ \delta(p, 0) = \delta(q, 1) = p $ and $ \delta(q, 0) = \delta(p, 1) = q $.

\end{example}

We will sometimes use automata diagrams in which some nodes without outgoing arrows are labelled by the names of other automata. This is shorthand and such a node should be mentally replaced by the start node of its label. See Figure~\ref{figure: automata concatenation} for an example.

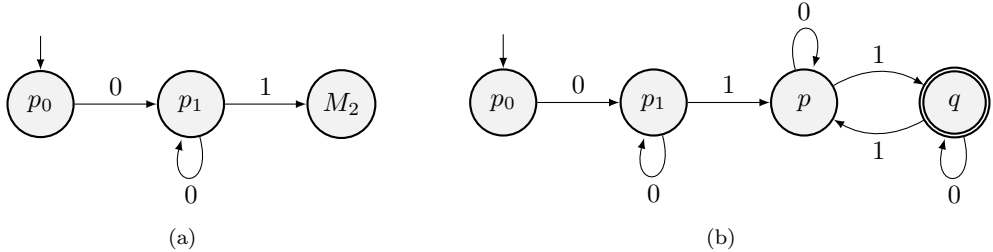
\begin{figure}[htbp]
        \centering  
        \subfigure[]{
        \begin{tikzpicture}
\node[state, initial above](root) {$ p_0 $};
\node[state] at (2, 0) (first) {$ p_1 $};
\node[state] at (4, 0) (second) {$ M_2 $};
\draw (root) edge[above] node{0} (first);
\draw (first) edge[above] node{1} (second);
\draw (first) edge[loop below, below] node{0} (first); 
      \end{tikzpicture}}\hspace{1cm}
      \subfigure[]{
      \begin{tikzpicture}
\node[state, initial above](root) {$ p_0 $};
\node[state] at (2, 0) (first) {$ p_1 $};
\node[state] at (4, 0) (second) {$ p $};
\draw (root) edge[above] node{0} (first);
\draw (first) edge[above] node{1} (second);
\draw (first) edge[loop below, below] node{0} (first);  
\node[state, accepting] at (6, 0) (rt) {$ q $};
\draw (second) edge[bend left, above] node{1} (rt);
\draw (rt) edge[bend left, below] node{1} (second);
\draw (second) edge[loop above] node{0} (second);
\draw (rt) edge[loop below] node{0} (rt);
\end{tikzpicture}}\caption{Diagram (a) is shorthand for Diagram (b)}\label{figure: automata concatenation}
        \end{figure}

An automaton running on an input word reads it character-by-character from left to right, while continuously updating its own state using $ \delta $. It \emph{accepts} the word if this run ends in an accept state.

\begin{defn}
Inductively define the partial function $ \deltahat : Q \times \Sigma^* \to Q $ for $ (q, x) \in Q \times \Sigma^* $ as
        \[\deltahat(q, x) :=
    \begin{cases}  q & \text{if $ x = \beps $};\\
    \delta(\deltahat(q, y), a) & \text{if $ x = ya $ for some $ (y, a) \in \Sigma^* \times \Sigma $, and $ \deltahat(q, y) $, $ \delta(\deltahat(q, y), a) $ are defined};\\
    \text{not defined} & \text{otherwise}.
    \end{cases}\]

    The \emph{run} of $ M $ on input $ x $ ends in the state $ \deltahat(q_s, x) $. An automaton $ M $ is said to \emph{accept} a word $ x $ if $ \deltahat(q_s, x) $ is defined and contained in $ F $. The set $L(M):= \{x \in \Sigma^* \mid M $ accepts $ x \} $ is the \emph{language} of $ M $.
\end{defn}

\begin{rem}\label{rem: good automata}
An automaton $ M = (Q, q_s, F, \delta) $ is said to be \emph{good} if for each state $ q \in Q $  there exist strings $ x, y $ with $ \deltahat(q_s, x) = q $ and $ \deltahat(q, y) \in F $. 
Given any automaton, removing the states which do not meet this condition yields a good automaton with the same language, so we will work without loss of generality with {good automata}.
\end{rem}

The connection between automata and string algebras was documented earlier in \cite{rees2008automata}, where it was shown that the set of strings for a string algebra is the language of a finite automaton whose alphabet is the set of all arrows of the quiver along with their inverses.

\begin{customthm}{$ \beta $}[\cite{rees2008automata}]
\label{thm: beta}
    For a string algebra $ \Lambda $, there is an automaton $ M_\Lambda $ over $ Q_1 \sqcup Q_1^- $ such that $ \St(\Lambda) = L(M_\Lambda) $.
\end{customthm}

However, \cite{rees2008automata} did not capture the hammock order relations between the strings. In this paper, we exploit a natural linear order relation $\inorder $, called \emph{inorder}, on the set $ \two^* $ of finite words over the alphabet $ \two $ ordered by $0<1$ to assign a linear order to every automaton.

Observe that $ \two^* $ has a natural tree structure, with $ x0 $ and $ x1 $ being respectively the left and right children of $ x $ for every word $ x $ (see Figure~\ref{fig: infinite binary tree}).
The inorder $ \inorder $ is the linear order relation which arranges the nodes of this tree from left to right.
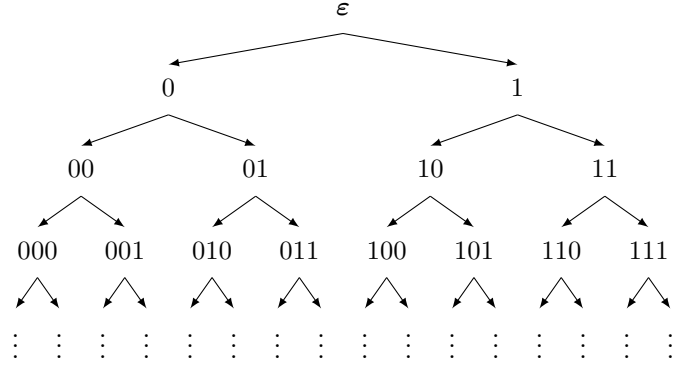
\begin{figure}[htbp]
  \begin{center}
\begin{forest}
    [$ \beps $ [0 [00 [000 [$ \vdots $] [$ \vdots $]] [001 [$ \vdots $] [$ \vdots $]]] [01 [010 [$ \vdots $] [$ \vdots $]] [011 [$ \vdots $] [$ \vdots $]]]] [1 [10 [100 [$ \vdots $] [$ \vdots $]] [101 [$ \vdots $] [$ \vdots $]]] [11 [110 [$ \vdots $] [$ \vdots $]] [111 [$ \vdots $] [$ \vdots $]]]]]
\end{forest}
\end{center}
  \caption{The inherent tree structure of $ \two^* $}
  \label{fig: infinite binary tree}
\end{figure}

\begin{defn}
  Given distinct words $ x, y \in \two^* $, we say that $ x \inorder y $ if  for the longest word $ w $ with $ x = wx' $ and $ y = wy' $ for some words $ x', y' $, either $ x' $ begins with a $ 0 $ or $ y' $ begins with a $ 1 $.
\end{defn}

The language of an automaton becomes a linear order under the induced order structure. 

\begin{defn}
\label{defn: linguages}
  The \emph{linguage} $ \lin(M) $ of an automaton $ M $ over $ \two $ is the linear order $(L(M), \inorder)$.
\end{defn}

\begin{example}
    Consider the automata $ M_1, M_2 $ from Example~\ref{example: finite automata}. Then $ L(M_1) = \{0^n1 \mid n > 0\} $ so that $ \lin(M_1) \cong \omega^* $. Similarly, $ L(M_2) = \{x \in \two^* \mid x $ has an even number of $ 1 $s$ \} $ so that $ \lin(M_2) \cong \eta $.
\end{example}

A fascinating connection between finite description linear orders and linguages of finite automata lies at the heart of this paper.

\begin{thm}
\label{thm: automata and linear orders}
  A linear order $ L $ is a finite description linear order if and only if it is order isomorphic to the linguage $ \lin(M) $ of some automaton $ M $ over $ \two $.
\end{thm}

This appears to be well-known--see \cite[Theorem~38]{bloom2005equational}, where the authors attribute it to Heilbr\"unner and Courcelle. The backward implication indeed follows from Theorem~\ref{thm: delta} due to them, {though after a little work (Construction~\ref{cons: gamma}, Theorem~\ref{thm: gamma})}. We could not find an explicit proof of the forward implication either, which is the content of the next result.

\renewcommand{\thesubfigure}{(\Roman{subfigure})}
\begin{figure}[htbp] 
\centering 
\subfigure[]{\scalebox{.95}{
\begin{tikzpicture}
\node[state, initial above](q1) {};\end{tikzpicture}}}
\subfigure[]{\scalebox{.95}{
\begin{tikzpicture}
\node[state, initial above, accepting](q1) {};\end{tikzpicture}}}
\subfigure[]{\scalebox{.95}{
\begin{tikzpicture}
\node[state, initial above](root) {};
\node[state] at (-1, -2) (left) {$ M_{L_0} $};
\node[state]  at (1, -2) (right) {$ M_{L_1} $};
\draw (root) edge[left] node{0} (left);
\draw (root) edge[right] node{1} (right);
\end{tikzpicture}}}
\subfigure[]{\scalebox{.95}{
\begin{tikzpicture}
\node[state, initial above](q1) {};
\node[state, below of=q1] (q2) {$ M_{L_0} $};
\draw (q1) edge[loop right] node{1} (q1);
\draw (q1) edge[right] node{0} (q2);
\end{tikzpicture}}}
\subfigure[]{\scalebox{.96}{
\begin{tikzpicture}
\node[state, initial above](q1) {};
\node[state, below of=q1] (q2) {$ M_{L_0} $};
\draw (q1) edge[loop left] node{0} (q1);
\draw (q1) edge[right] node{1} (q2);
\end{tikzpicture}}}
\subfigure[]{

\scalebox{.7}{
\begin{tikzpicture}
    \node[state, initial above](m) {};
    \node[state] at (-3, 1) (l1) {};
    \node[state] at (-2, -1) (l2) {};
    \node[state] at (3, 1) (r1) {};
    \node[state] at (2, -1) (r2) {};
    \node[state] at (-3, -2.7) (l3) {$ M_{L_0} $};
    \node[state] at (3, -2.7) (r3) {$ M_{L_1} $};
    
   \draw (m) edge[bend right, above] node{0} (l1);
    \draw (m) edge[bend left, above] node{1} (r1);
    
    \draw (l1) edge[bend right, left] node{1} (l2);
    \draw (r1) edge[bend left, right] node{0} (r2);
    
    \draw (l1) edge[bend right, below] node{0} (m);
    \draw (r1) edge[bend left, below] node{1} (m);
    
    \draw (l2) edge[bend right, below] node{1} (m);
    \draw (r2) edge[bend left, below] node{0} (m);

    \draw (l2) edge[left] node{0} (l3);
    \draw (r2) edge[right] node{1} (r3);
    
\end{tikzpicture}}}

\caption{(I) $ M_\mbf 0 $, (II) $ M_\mbf 1 $, (III) $ M_{{L_0} + {L_1}} $, (IV) $ M_{{L_0} \cdot \omega} $, (V) $ M_{{L_0} \cdot \omega^*} $, (VI) $ M_{\Xi({L_0}, {L_1})} $}
\label{fig:automata}
\end{figure}
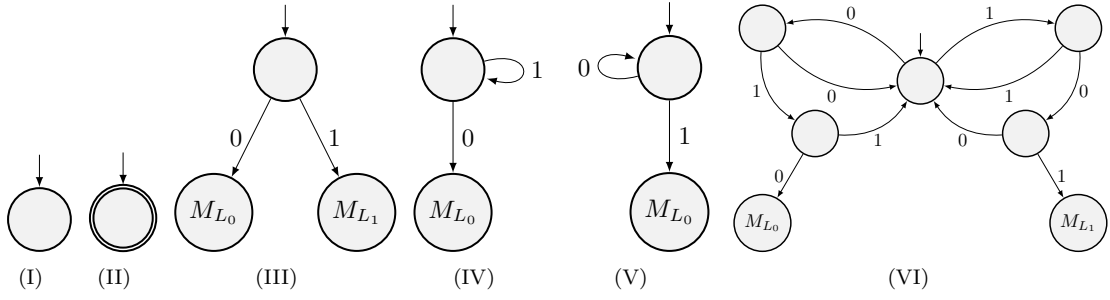
\begin{customconsthm}{$ \epsilon $}
\label{consthm: epsilon}
\label{prop: automata for LOfd}

Let $ L_0, L_1 $ be the linguages of finite automata $ M_{L_0}, M_{L_1} $ respectively.
For each $ J \in \{\mbf 0, \mbf 1, {L_0} + {L_1}, {L_0} \cdot \omega, {L_0} \cdot \omega^*, \Xi({L_0}, {L_1}) \} $, let $ M_J $ be as in Figure \ref{fig:automata}. Then the following hold.
\vspace{2pt}
\begin{minipage}{0.25\textwidth}
    
\begin{enumerate}[label=(\Roman*)]
    \item $ \lin(M_\mbf 0) \cong \mbf 0 $.
    \item $ \lin(M_\mbf 1) \cong \mbf 1 $.
    \end{enumerate}
\end{minipage}
\begin{minipage}{0.3\textwidth}
\begin{enumerate}[label=(\Roman*)]
  \setcounter{enumi}{2}
  \item $ \lin(M_{{L_0} + {L_1}}) \cong {L_0} + {L_1} $.
    \item $ \lin(M_{{L_0} \cdot \omega}) \cong {L_0} \cdot \omega $
    \end{enumerate}
\end{minipage}
\begin{minipage}{0.35\textwidth}
    \begin{enumerate}[label=(\Roman*)]
    \setcounter{enumi}{4}
    \item $ \lin(M_{{L_0} \cdot \omega^*}) \cong L_0 \cdot \omega^* $.
    \item $ \lin(M_{ \Xi({L_0}, {L_1})}) \cong \Xi({L_0}, {L_1}) $.
    \end{enumerate}
\end{minipage}
\end{customconsthm}

To construct an automaton that performs $ n $-ary shuffles for arbitrary finite $ n \ge 1 $,  
take a binary tree with $ 2n + 1 $ leaf nodes $ v_0, v_1, \dots, v_{2n} $ from left to right, identify every even node with the root and label the $ (2i + 1)^{th} $ node $ M_i $, where $ M_i $ is the automaton for the $ i^{th} $ linear order being shuffled.

We now show that the linguage of any automaton is isomorphic to the linear order underlying a component of the universal solution of a word problem in the alphabet $ \{{\star}\} $. Let us outline the construction of the word problem associated with an automaton, together with a proof of its correctness for the sake of completeness.

\begin{customcons}{$ \gamma $}
\label{cons: gamma}
Let $ M $ be an automaton with $ n $ states $ q_1, q_2, \dots, q_n $ of which $ q_1 $ is the start state. The \emph{word problem associated to $ M $}, denoted $ W_M $, is the system $\lan v_i = v_{i0} b_i v_{i1}\mid 1\leq i\leq n \ran$ of equations over $ \Gamma := \{{\star}\}$, 
    where for any $1\leq i \leq n$ and $ d \in \{0, 1\} $, 
    \begin{align*}
        v_{id} & := \begin{cases}
            v_j & \text{if $  \delta(q_i, d) $ is defined to be equal to $ q_j $};\\
            \beps & \text{otherwise,}
        \end{cases} \text{ and \hspace{4pt}}      b_i := \begin{cases}
            {\star} & \text{if $ v_i $ is an accept state};\\
            \beps & \text{otherwise.}
        \end{cases}
    \end{align*}
\end{customcons}
\begin{figure}[htbp]
        \centering
        \begin{tikzpicture}
    \node[state,  accepting, initial left] at (0, 2)(root)  {$ 1_{(v, 1)} $};
\node[state, accepting] at (2, 2) (first) {$ b^- $};
\node[state, accepting] at (4, 2) (second) {$ b^-b^- $};
\node[state, accepting] at (2, 0) (third) {$ ab^- $};
\node[state, accepting] at (0, 0) (left) {$ a $};
\node[state, accepting] at (4, -1) (fourth) {$ b^-a $};
\draw (root) edge[above] node{1} (first);
\draw(root) edge[left] node{0} (left);
\draw(first) edge[left] node{0} (third);
\draw(second) edge[above] node{0} (third);
\draw(fourth) edge[right] node{1} (second);
\draw(third) edge[above, bend right] node{1} (fourth);
\draw(fourth) edge[above, bend right] node{0} (third);
\draw (first) edge[above] node{1} (second);
\draw(left) edge[below, bend right] node{1} (fourth);
\end{tikzpicture}
\caption{Automaton $ M' $ (see Example~\ref{example: beta' on gp23} for an explanation on the labelling of nodes)}
        \label{fig: beta' on gp23}
    \end{figure}
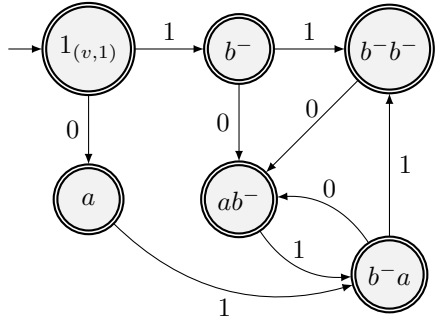
\begin{example}
Applying Construction~\ref{cons: gamma}  to the automaton in Figure~\ref{fig: beta' on gp23} gives us the following word problem over the alphabet $ \Gamma = \{{\star}\} $ in six unknowns: $ \lan u = u'{\star} v, v = u'{\star}, w = w'{\star} u, u' = {\star} v', v' = u'{\star} v, w' = {\star} v' \ran $. A few substitutions give us $  v' = u $, $ u'=w' = {\star} u $, $ v = {\star} u {\star} $, $ w = {\star} u {\star} u $, and $ u = {\star} u{\star} {\star} u{\star} $. The universal solution given by Heilbr\"unner's algorithm has $ u = (\omega + \zeta \cdot \eta + \omega^*, \chi:\omega + \zeta \cdot \eta + \omega^* \to \{{\star}\}) $.
\end{example}

The next result is the proof of correctness of Construction \ref{cons: gamma}.

\begin{customthm}{$ \gamma $}
    \label{thm: gamma}
    Continuing with the notations in Construction~\ref{cons: gamma}, let $ \lan A_1, A_2, \dots, A_n\ran $ be the universal solution to $ W_M $ and let $ L_1 $ be the linear order underlying $ A_1 $. Then $ L_1 \cong \lin(M)$.
\end{customthm}
\begin{proof}
	Let $ g : \two^* \to \omega^{< \omega} $ be the unique concatenation-preserving map sending $ 0 $ to $ 0 $ and $ 1 $ to $ 2 $.

	Let $ (T, \chi:T\to \{{\star}\}\sqcup\{v_1,v_2,\dots,v_n\}) $ be the unique labelled tree such that the following hold for each $ y \in  T, x \in \two^* $ and $ 1 \le i \le n $:
 
    \begin{minipage}{0.5\textwidth}
	\begin{itemize}
		\item $ \chi(\beps) = v_1 $.
            \item $ \deltahat(q_1, x) $ is defined $ \iff g(x) \in T $.
            \item $ x \in L(M) \iff g(x)1 \in T $. 
            
            \end{itemize}
            \end{minipage}%
    \begin{minipage}{0.5\textwidth}
            \begin{itemize}
            
		\item $ y \in \{g(z)1, g(z)\} $ for some $ z \in \two^* $.
            \item $ \deltahat(q_1, x) = q_i \implies  \chi(g(x)) = v_i $
            \item $ x \in L(M) \implies \chi(g(x)1) = {\star} $. 
	\end{itemize}
 \end{minipage}

	Let $ (T_1, \chi_1) $ be the $ \{{\star}\} \sqcup \{v_1, v_2, \dots, v_n\} $-labelled tree corresponding to $ v_1 $ produced upon applying Construction~\ref{cons: delta} to the word problem $ W_M $. 
	Notice that the root of $ T $ is labelled $ v_1 $ and every node labelled $ v_i $ has left child $ v_{i0} $ or none, if nonexistent, right child $ v_{i1} $ or none, if nonexistent, and middle child $ {\star} $, if $b_i = {\star} $, or none, if not. 
 Thus, it is clear that there exists a label-preserving, $ \subseteq $-preserving, and $ \preorder $-preserving bijection $ (T, \chi) \to (T_1, \chi_1) $. Thus $ (\Front(T), \chi) \cong (\Front(T_1), \chi_1) \cong L_1 $, where the second isomorphism is due to Theorem~\ref{thm: courcelle frontier}.
   It suffices to construct a function $  f_M : \two^* \to \omega^{<\omega} $ with $ f_M(L(M)) = \Front(T) $ so that the induced map $ f_M : \lin(M) \to (\Front(T), \preorder) $ is order-preserving.

    Let $ f_M $ be defined by $f_M(x):= \begin{cases}
        g(x) 1 & \text{if $ \delta(q_s, x) $ is defined and is an accept state};\\
        g(x) & \text{otherwise.}
    \end{cases}$
    
    We claim that $ f_M $ is as required.
    There are three things to check.
    \begin{itemize}
        \item $ \Front(T) \subseteq f_M(L(M)) $: If $ y \in \Front(T) $, then it must be the case that $ y = g(x)1 $ for some $ x \in \two^* $. For otherwise, $ y = g(x) $ for some $ x \in \two^* $. Let $ q_i = \deltahat(q_s, x) $. Using our assumption that the automata we work with are good (Remark~\ref{rem: good automata}), choose and fix a word $ z \in L(M) $ extending $ x $, i.e., either $ z = x $ so that $ g(x)1 \in T $, or $ x \subsetneq z $ and $ y = g(x) \subsetneq g(z) \in T $. In either case, our supposition that $ y = g(x) $ is a leaf node is contradicted. 
        Thus $ y = g(x) 1 $ for some $ x \in \two^* $, which means that $ x \in L(M) $, $ f_M(x) = g(x)1 = y $ and $ y \in f_M(L(M)) $.

        \item $ f_M(L(M)) \subseteq \Front(T) $: For any  $ x \in L(M) $, $ f_M(x) = g(x)1 \in T $. If $ f_M(x) \not\in \Front(T) $, then there must exist a proper extension $ y $ of $ f_M(x) $ in $ T $. Either $ y = g(z) $ or $ y = g(z)1 $ for some $ z \in \two^* $. In either case $ f_M(x) \subseteq g(z) $, contradicting the fact that $ f_M(x) $ ends in a $ 1 $.
        
        \item $ f_M $ is order-preserving: Let $ x, y \in L(M) $ with $ x \inorder y $ and fix the longest $ w $ for which $ x = wx' $ and $ y = wy' $ for some $ x', y' $. By the definition of $ \inorder $, either $ x' $ begins with a $ 0 $, or $ y' $ begins with a $ 1 $. Without loss of generality, assume the former, so $ y' $ is either $ \beps $ or begins with a $ 1 $. Then $ x = w0x'' $ and $ y = wy' $, so $ f_M(x) = g(w)0g(x'')1 $ and $ f_M(y) = g(w)g(y') $, where $ g(y') $ is either empty or begins with a $ 2 $. Thus $ g(w) $ is the longest word with $ f_M(x) = wf $ and $ f_M(y) = wf' $ and $ f $ begins with a $ 1 $, so that $ f_M(x) \preorder f_M(y) $.
    \end{itemize}
\end{proof}

It is possible to detect whether an automaton has a finite or scattered linguage without having to calculate the latter completely.

\begin{defn}
    Given a state $ p $ in an automaton $ M $ over $ \two $, say that $ p $ is 
    \begin{itemize}
        \item an $ \omega $-state if there exists a non-empty string $ x $ such that $ \deltahat(p, x) = p $;
        \item an $ \overline\omega $-state if there is a string $ x $ such that $ \deltahat(p, x) $ is an $ \omega $-state;
        \item an $ \eta $-state if there exist  strings $ x_0, x_1 $ such that $ \deltahat(p, 0x_0) = \delta(p, 1x_1) = p $; and
        \item an $ \overline\eta $-state if there is a string $ x $ such that $ \deltahat(p, x) $ is an $ \eta $-state.
    \end{itemize}
\end{defn}
\begin{rem}
    Note that ($ p $ is a  $ \xi $-state  $ \implies  p $ is a $ \zeta $-state) holds for $ (\xi, \zeta) \in \{(\eta, \overline \eta), (\overline \eta, \overline \omega), (\eta, \overline \omega), (\eta, \omega), (\omega, \overline \omega)\} $.
\end{rem}

We will use the next result in a later section.
\begin{prop}
\label{prop: etabar and omegabar}
    Let $ M $ be an automaton with start state $ q_s $. Then $ \lin(M) $ is 
    \begin{itemize}
        \item non-scattered if and only if $ q_s $ is an $ \overline\eta $-state.
        \item infinite if and only if $ q_s $ is an $ \overline\omega $-state.
    \end{itemize}
\end{prop}
Theorem~\ref{prop: quasi-rational systems} yields the first assertion while the second follows from the pigeonhole principle. 

\section*{$ \beta' : $ Hammocks as linguages}
\label{subsection: hammocks as linguages}
\label{section: beta'}
Fix a string algebra $ \Lambda = \mathcal K \mathcal Q / \lan \rho \ran $ and a vertex $ v \in Q_0 $. 
We present the construction of an automaton $ M'_\Lambda(v) $ over $ \two $ with $ \lin(M'_\Lambda(v)) \cong H_l(v) $.
For this, we start by associating to every string $ \xx \in H_l(v) $ a binary word $ \sgn(\xx) $, called its \emph{sign sequence}.

\begin{defn}
  The function $ \sgn : H_l(v) \to \two^* $ is defined recursively by \[ \sgn(\xx) := \begin{cases} 
    \beps & \text{if $ \xx=1_{(v,1)}$;}\\ 
    \sgn(\yy)0 & \text{if $ \xx = \alpha\yy $ for $ \alpha \in  Q_1 $;}\\ 
    \sgn(\yy)1 & \text{if $ \xx = \alpha\yy $ for $ \alpha \in Q_1^{-}$}. 
    \end{cases}\]
\end{defn}
Notice that even though a string grows on the left, the function $\sgn$ grows the corresponding word on the right. The proof of the next result is straightforward.
\begin{prop}
\label{prop: sgn preserves order}
The map $ \sgn : (H_l(v), <_l) \to (\two^*, \inorder) $ is an order embedding, i.e., it preserves and reflects order relations .
\end{prop}

We now construct an automaton with language $ \sgn(H_l(v)) \subseteq \two^* $. The automaton will simulate walking on the quiver $ \mathcal Q $ along arrows or their inverses, while storing sufficiently many syllables in its memory to be able to detect the formation of blocking relations.
A key role is played here by the finiteness of the set $ \rho $ of blocking relations.

\def\rmshort{\mathrm{short}}

\begin{defn}
    Let $ \short(\Lambda) \subseteq \St(\Lambda) $ be the set of \emph{short strings}, i.e., those strictly shorter than the length of the longest relation in $ \rho $. 
    For any string $ \xx \in \St(\Lambda) $, let $  [\xx] $ denote the longest string $ \yy \in \short(\Lambda) $ such that $ \xx = \yy \zz $ for some string $ \zz $.
    Let $ \rmshort(\Lambda) := \abs{\short(\Lambda)} $.
\end{defn}

\begin{rem}
\label{rem: short}
Even though $ \St(\Lambda) $ could be infinite, $\short(\Lambda)$ is finite (with ${\rmshort(\Lambda)} \le \abs{Q_0}(2^{r+2} - 1) $, where $r$ is the length of the longest relation in $\rho$), and will serve as the set of states for our automaton.
\end{rem}

\begin{customcons}{$ \beta' $}
\label{defn: our automaton}
\label{cons: beta'}
    Let $ M'_\Lambda(v) $ be the automaton over $\two$ given by the data $ (\short(\Lambda), 1_{(v, 1)}, \short(\Lambda), \delta) $, where $ \delta : \short(\Lambda) \times \two \to \short(\Lambda) $ is defined as
    \[\delta(\xx,c) := \begin{cases}
          [\beta\xx] &\text{if there exists $ \beta \in Q_1\sqcup Q_1^- $ with $\sgn(\beta)=c$ and $ \beta\xx \in \St(\Lambda) $};\\
        \text{undefined} &\text{otherwise}.
    \end{cases}\]
\end{customcons}

The partial map $ \delta $ is well defined because any such $ \beta $, if it exists, is unique.

\begin{example}
\label{example: beta' on gp23}
Construction~\ref{cons: beta'} applied to the data consisting of the string algebra $ GP_{2, 3} $ from Example~\ref{example: string algebra} together with the only vertex $v$, yields the automaton in Figure~\ref{fig: beta' on gp23}, where each  node represents the string by which it is labelled. 
\end{example}

We now show that this automaton indeed accepts only the sign sequence of strings in $ H_l(v) $.
\begin{thm}
\label{thm: automaton accepts sgns of strings}
  For any $ x \in \two^* $, $ x\in L(M'_\Lambda(v))$ if and only if there is $ \xx \in H_l(v) $ with $ \sgn(\xx) = x $.
\end{thm}

\begin{proof}
For the backward direction, a simple induction on the length of $\xx\in H_l(v)$ shows that $\deltahat(1_{(v,1)}, \sgn(\xx))=[\xx]$, and hence $\sgn(\xx)\in L(M'_\Lambda(v))$. 

  For the forward direction, suppose for the sake of contradiction that $ x \in L(M'_\Lambda(v)) $ is of minimal length such that there is no $ \xx \in H_l(v) $ with $ \sgn(\xx) = x $. Clearly $ x \not= \beps $, so $ x = yu $  for some $ u \in \two $.
  The minimality of $ x $ then implies that there is $ \yy \in H_l(v) $ with $ \sgn(\yy) = y $. Since $ \deltahat(1_{(v, 1)}, x) = \delta(\deltahat(1_{(v, 1)}, y), u) $ is defined, there exists a unique $ \beta \in Q_1 \sqcup Q_1^- $ with $ \beta\yy \in \St(\Lambda) $ and $ \sgn(\beta) = u $. But then $ \sgn(\beta\yy) = yu=x $, contradicting our choice of $ x $, thereby completing the proof.
\end{proof}

The next result is the combination of Proposition \ref{prop: sgn preserves order} and Theorem \ref{thm: automaton accepts sgns of strings}, and is the proof of correctness of Construction \ref{cons: beta'}.
\begin{customthm}{$ \beta' $}
\label{cor: hammocks as linguages}
\label{thm: beta'}
   The map $ \sgn : (H_l(v), <_l) \to \lin(M'_\Lambda(v)) $ is an order-preserving bijection. 
\end{customthm}
The above result together with Proposition \ref{prop: etabar and omegabar} yields the following.
\begin{cor}
The algebra $\Lambda$ is a non-domestic string algebra if and only if the start state of the automaton $M'_\Lambda(v)$ is an $\overline\eta$-state for some $v\in Q_0$.    
\end{cor}

\begin{rem}\label{complexity alpha}
Order types of hammock linear orders can be explicitly computed in $ O(\rmshort(\Lambda))$ time (Remark~\ref{rem: delta}) by using Heilbr\"unner's algorithm to compute the initial solution to the word problem $ W_{M'_\Lambda(v)} $ produced by applying Construction~\ref{cons: gamma} to the automaton $ M'_\Lambda(v) $. 
The relation $\alpha=\delta\gamma\beta'$ in the quiver in Figure~\ref{fig: heart of intro} is therefore justified as Theorems \ref{thm: delta}, \ref{thm: gamma} and \ref{thm: beta'} together provide an alternate proof of Theorem~\ref{thm: alpha}. 
\end{rem}

\begin{rem}\label{rem: gen hamm}
Given a string $\xx_0$, we can change the start state of $M'_\Lambda(v)$ to $ [\xx_0]$ so as to obtain the automaton $ M'_\Lambda(\xx_0) $ satisfying $\lin(M'_\Lambda(\xx_0))\cong (\{\yy \in \St(\Lambda) \mid \exists \zz\in\St(\Lambda)(\yy = \zz\xx_0)\}, <_l) =: (H_l(\xx_0), <_l) $.    
\end{rem}

\begin{rem}
By symmetry, the right hammock $(H_r(v),<_r)$ too is computable using finite automata. By the product construction for automata, it is easy to construct an automaton over the alphabet $\two\sqcup\two = \{0_l, 1_l, 0_r, 1_r\} $ which accepts only the strings of the form $ x_lx_r $ with $ x_l \in \{0_l, 1_l\}^* $ and $ x_r \in \{0_r, 1_r\}^* $ representing the sign sequences of $ \xx_l, \xx_r $ with $ (\xx_l, \xx_r) \in H(v) $. The partial order relation $ < $ is captured as the order induced on this language by a suitably defined partial order relation $ <_{\two \sqcup \two} $ on $ (\two \sqcup \two)^* $.
\end{rem}

\section*{$ E $ : Exceptional bands}
\label{section: E}

In this section, we give an automata-theoretic proof of Theorem~\ref{thm: E}, i.e., the finiteness of the set of exceptional points, and describe an algorithm to compute this set (Remark~\ref{rem: E}), the size of which determines the error term $e$ in the stable rank computation in Theorem \ref{thm: Special}.

Here we will only talk about the intervals around band points in the left hammock $\overline H_l(v)$; a similar notion can also be defined for the right hammocks $\overline H_r(v)$.
\begin{defn}
\label{defn: exceptional bands}
Say that a band $ \bb' $ for $\Lambda$ is \emph{right exceptional} (resp. \emph{left exceptional}) if, for some $v\in Q_0$, $\xx\in H_l(v)$ and some cyclic permutation $\bb$ of $\bb'$, the interval $ (\infb, \xx] $ (resp. $ [\xx, \binf)$) in $\overline H_l(v)$ is scattered. Say that $ \bb' $ is \emph{exceptional} if it is either left or right exceptional.
\end{defn}

Since the contribution of each exceptional band point to the error term $e$ is at most $ 1 $, by symmetry it is sufficient to show that the set of left exceptional band points is finite. The following lemma will allow us to exploit the finiteness of the automata for hammock computation.
\begin{lem}
\label{lem: exceptional bands automatically}

Let $ \bb \in H_l(v) $ be a cyclic permutation of a band $ \bb' $ and, for brevity, let $ \yy $ denote $ \infb $, $ \yy(n) $ the $ n^{th} $ syllable of $ \yy $, and $ \yy_n $ the string formed by the first $ n $ syllables of $ \yy $. 
Then $ \bb' $ is left exceptional if and only if there is $ N\in\mathbb N$ such that for every $ n \ge N $ with $ \delta([\yy_n], 0) $ an $ \overline \eta $-state, the syllable $ \yy(n) $ is direct.
\end{lem}

\begin{proof}

$(\Rightarrow)$
Since the number of states in the automaton $M'_\Lambda(v)$ that computes the order type of $H_l(v)$ is finite, the pigeonhole principle yields $N, k\in\mathbb N$ such that $\yy(n) = \yy(n+k) $ and $ [\yy_n] = [\yy_{n + k}]$ for every $ n \ge N $.

To prove by contradiction that this $ N $ is as required, suppose there is $ n \ge N $ for which $\delta([\yy_n], 0) $ is an $ \overline \eta $-state but $ \beta :=\yy(n)\in Q_1^{-}$. Then there also exists $ \alpha \in Q_1 $ such that $\alpha\yy_n\in\St(\Lambda)$. Then for every $ m > n $, the set $ I_m $ of strings extending $ \alpha(\yy(n+k-1)\dots\yy(n+1)\beta)^m\yy_{n-1}$ on the left is a non-scattered interval in $ H_l(v) $ by Proposition~\ref{prop: etabar and omegabar}. These intervals $I_m$ are pairwise disjoint. Moreover, for any $ \zz \in H_l(v) $ with $ \zz <_l \yy$ in $\overline H_l(v)$ there exists $ m\in\mathbb N$ such that $ I_m\subseteq [\zz, \yy)$, and thus an embedding of $\eta$ in $[\zz,\yy)$-a contradiction to the hypothesis that $\bb$ is left exceptional.

$(\Leftarrow)$
\def\yym#1{\yy^-_{#1}}
Suppose that the given condition holds for $ N \in \N $. 
Let $ \lan m_n \mid n \in \N \ran $ be an increasing enumeration of the set $ \{m \in \N \mid \yy(m) $ is inverse$ \} $.

Visualizing $ H_l(v) $ as the binary tree formed by its image $ \sgn(H_l(v)) $ in $ \two^* $ under the map $ \sgn $, $ <_l $ as the inorder on this tree and $ \yy $ as an infinite branch in this tree, it is easy to see that $ \{\xx \in H_l(v) \mid \xx <_l \yy\} = \sum_{n\geq 1}I_n$, where $I_n := \{\xx \in H_l(v) \mid \xx = \xx'\yy_{m_n} \text{for some $ \xx' $ which does not have $ \yy(m_n) $ as its rightmost syllable} \}$. In fact, for each $ n $, we have $I_n=\begin{cases}
    H_l(\alpha_n \yy_{m_n}) + \mbf 1& \text{if there exists $ \alpha_n \in Q_1 $, necessarily unique, with $ \alpha_n \yy_{m_n} $ a string};\\
    \mbf 1 &\text{otherwise}.
\end{cases}$

It therefore suffices to show that there are only finitely many values of $ n $ for which $ I_n $ is non-scattered. For any such $ n $, we necessarily have $ I_n = H_l(\alpha_n\yy_{m_n}) + \mbf 1 \cong \lin(M'_\Lambda(\alpha_n\yy_{m_n})) + \mbf 1 $, where the isomorphism is by Remark~\ref{rem: gen hamm}. By Proposition~\ref{prop: etabar and omegabar}, the linguage $ \lin(M'_\Lambda(\alpha_n\yy_{m_n}))$ is non-scattered if and only if $ [\alpha_n\yy_{m_n}] = \deltahat(q_s, \sgn(\alpha_n\yy_{m_n})) = \delta(\deltahat(q_s, \sgn(\yy_{m_n})), \alpha_n) = \delta([\yy_{m_n}], 0) $ is an $ \overline\eta $-state in $ M'_\Lambda(v)$. Recall that for any $ n $, the syllable $ \yy(m_n) $ is inverse. Since $ \yy(m) $ is direct for every $ m \ge N $ with $ \delta([\yy_m], 0) $ an $ \overline\eta $-state, we have $ n \le m_n < N $. Thus there can only be finitely many such $ n $.
\end{proof}

The characterization of (left) exceptional bands provided by Lemma~\ref{lem: exceptional bands automatically} allows us to calculate their set algorithmically.

The sequence $ \lan([\yy_{N + j}], \sgn(\yy(N + j))) \mid j < k \ran $ is called the \emph{periodic part} of $ \yy=\infb $. It encodes the periodic sequence of transitions that the automaton $ M'_\Lambda(v) $ eventually enters while operating on the infinite input string $ \sgn(\yy) $. We are now ready to prove Theorem \ref{thm: finitely many exceptional bands}.

\begin{proof}[Proof of Theorem \ref{thm: finitely many exceptional bands}]
Because there are only finitely many pairs of the form $ (q, 0) $, where $ q $ is a state, it suffices to show that if $ \bb'_1, \bb'_2 $ are left exceptional bands whose periodic parts contain a common transition $ (q, 0) $, then they are identical. For $i=1,2$, let $\yy^i:=\infb_i$, where $\bb_i\in H_l(v)$ for appropriate cyclic permutations $\bb_i$ of $\bb'_i$. We use other notations similar to Lemma \ref{lem: exceptional bands automatically}.

Suppose $([\yy^1_{n_1}], \sgn(\yy^1(n_1))) = ([\yy^2_{n_2}], \sgn(\yy^2(n_2))) = (q, 0)$. For each $ n\in\mathbb N$ and $ i \in \{1, 2\}$, define $ p^i_n :=\sgn(\yy_i(n_i + n))$. We claim that $ p^1_n = p^2_n$ for every $ n $, from which the result follows.

Suppose the claim is false. Since $ p^1_0 = p^2_0 $, there exists a least positive $n\in\mathbb N$ for which $ p^1_{n + 1} \not= p^2_{n + 1} $.
Then $ q^1_{n_1 + n} = q^2_{n_2 + n} =: q' $ (say), and without loss of generality $ p^1_{n + 1} = 0 $ and $ p^2_{n +1} = 1 $. Since the periodic parts form cycles, it follows that $ q' $ is accessible from both $ q^1_{n_1 + n + 1} = \delta(q', 0) $ and $ q^2_{n_2 + n + 1} = \delta(q', 1)$ and hence an $\eta$-vertex, so $ \delta(q', 0) $ is an $ \overline\eta $-vertex. This is a contradiction to Lemma \ref{lem: exceptional bands automatically} applied to $\bb'_2$ at $ n_2 + n $, for $ \yy_2(n_2 + n) = 1 $ despite $ \delta(q^2_{n_2 + n}, 0) = \delta(q', 0) $ being an $ \overline\eta $-vertex.
\end{proof}

\begin{defn}
    
Let $ \mathcal{Q}'   $ be the quiver underlying the diagrammatic presentation of the automaton $ M(v) $, so that $ \mathcal{Q}' $ has $ \short(\Lambda) $ as its set of vertices and the graph of $ \delta $ as its set of arrows. Consider any edge $ ((q_1, d), q_2) $ to have source $ q_1 $, target $ q_2 $.

Let $ D:=\{((q_1, 1), q_2) \in\mathrm{Graph}(\delta) \mid  ((q_1, 0), q_3) \in\mathrm{Graph}(\delta)$ for some $ \overline\eta $-state $ q_3 \} $, and define $ \mathcal{Q}'' $ to be the subquiver of $ \mathcal{Q}' $ obtained by deleting the vertices corresponding to non-$ \overline\omega $-states, their incoming and outgoing arrows, and arrows in $ D $.
\end{defn}

The proof of Lemma~\ref{lem: exceptional bands automatically} characterises (left) exceptional bands for $ \Lambda $ as cycles in the quiver $ \mathcal{Q}'' $. Theorem~\ref{thm: E} uses the fact that any two distinct such cycles cannot share an arrow of the form $((q_1,0),q_2)$, thus bounding the number of exceptional bands by the number of such arrows. This naturally yields an algorithm to compute the set of (left) exceptional bands for $ \Lambda $.

\begin{rem}
\label{rem: E}
 The computation of $ \mathcal{Q'} $ takes $ O(\rmshort(\Lambda)) $ time and space, after which, the sets of $ \overline\eta $ and $ \overline\omega $-states can be computed in $ O(\rmshort(\Lambda)^w) $ time  by computing  the transitive closure of the arrow set of $ \mathcal Q' $, where $ w $ is the smallest real for which an $ O(n^w) $ time algorithm for  $ n \times n $ matrix multiplication is known.
Having thus constructed $ \mathcal{Q}'' $, it  suffices to compute the set of all cycles in this quiver. Since each cycle contains an arrow of the form $ ((q_1, 0), q_2) $ and each such arrow appears in at most one cycle, cycles in $ \mathcal{Q}'' $ are precisely the strongly connected components, which can be computed in $ O(\rmshort(\Lambda)) $ time by the Kosaraju-Sharir algorithm \cite{sharir1981strong}.
Thus the total time complexity of the algorithm is $ O(\rmshort(\Lambda)^w) $.
\end{rem}

This also implies that the number of  exceptional points, and hence the error term $ e $ in stable rank computation, is bounded above by the number of arrows in $ \mathcal{Q}'' $. In combination with the bound on $ \rmshort(\Lambda) $ given in Remark~\ref{rem: short}, this gives $ e < 2\cdot\rmshort(\Lambda) < \abs{Q_0}(2^{r+3} - 2) $.

\printbibliography

\end{document}